\documentclass[a4paper,12pt]{amsart}
\usepackage{probpack}
\usepackage[margin=3truecm,top=4truecm]{geometry}

\begin{document}
\title[Concentrations for simple random walk in random potential]
	{Concentrations for the simple random walk in unbounded nonnegative potentials}
\author[N. KUBOTA]{Naoki KUBOTA}
\address[N. Kubota]{College of Science and Technology, Nihon University,
	Chiba 274-8501, Japan.}
\email{kubota.naoki08@nihon-u.ac.jp}

\begin{abstract}
We consider the simple random walk in i.i.d.\,nonnegative potentials on the multidimensional cubic lattice.
Our goal is to investigate the cost paid by the simple random walk for traveling from the origin to a remote location
in a landscape of potentials.
In particular, we obtain concentration inequalities for the travel cost in unbounded nonnegative potentials.
\end{abstract}

\maketitle

\section{Introduction}

\subsection{The model}
Let $(S_k)_{k=0}^\infty$ be the simple random walk on the $d$-dimensional cubic lattice $\Z^d$, $d \geq 2$.
For $x \in \Z^d$, write $P^x$ for the law of the random walk starting at $x$,
and $E^x$ for the associated expectation.
Furthermore, we consider the measurable space $\Omega:= [0,\infty)^{\Z^d}$
endowed with the canonical $\sigma$-field $\mathcal{G}$.
Let $\P$ be the corresponding product measure on $(\Omega,\mathcal{G})$
and denote an element of $\Omega$ by $\omega=(\omega(x))_{x \in \Z^d}$,
which is called the potential.
To avoid trivialities we suppose that $\omega(0)$ is not almost surely equal to $0$.

For a subset $V$ of $\R^d$, $H_V$ stands for the hitting time of $(S_k)_{k=0}^\infty$ to $V$, i.e.,
\begin{align*}
	H_V:=\inf \{ k \geq 0; S_k \in V \},
\end{align*}
and denote $H(y):=H_{\{ y \}}$ for $y \in \Z^d$.
Furthermore we define for $x,y \in \Z^d$,
\begin{align*}
	e(x,y,\omega ):=E^x \Biggl[ \exp \Biggl\{ -\sum_{k=0}^{H(y)-1} \omega (S_k) \Biggr\} \1{\{ H(y)<\infty \}} \Biggr],
\end{align*}
where $e(x,y,\omega):=1$ if $x=y$.
Let us now introduce the travel cost $a(x,y,\omega)$ from $x$ to $y$
for the simple random walk in a potential $\omega$ as follows:
\begin{align*}
	a(x,y,\omega ):=-\log e(x,y,\omega),\qquad x,y \in \Z^d.
\end{align*}
Throughout this paper, we drop $\omega$ in the notation if there is no confusion.

Note that the strong Markov property shows the subadditivity
\begin{align*}
	a(x,z) \leq a(x,y)+a(y,z),\qquad x,y,z \in \Z^d,
\end{align*}
see Proposition~2 of \cite{Zer98a} for more details.
Since we now treat i.i.d.\,potentials, the subadditive ergodic theorem shows
the following essential asymptotic for the travel cost.
For the proof, we refer the reader to \cite[Proposition~4]{Zer98a}.

\begin{thm}[Zerner]\label{thm:zer}
Assume $\E[\omega(0)]<\infty$.
Then there exists a norm $\alpha(\cdot)$ on $\R^d$ (which is called the Lyapunov exponent) such that
for all $x \in \Z^d$, $\P \hyphen \as$ and in $L^1(\P)$,
\begin{align}\label{eq:lyapunov}
	\lim_{n \to \infty} \frac{1}{n}a(0,nx)
	= \lim_{n \to \infty} \frac{1}{n}\E[a(0,nx)]
	= \inf_{n \geq 1} \frac{1}{n}\E[a(0,nx)]
	= \alpha(x).
\end{align}
Furthermore, $\alpha(\cdot)$ is invariant under permutations of the coordinates and under reflections
in the coordinate hyperplanes, and satisfies
\begin{align}\label{eq:lbd}
	-\log \E[e^{-\omega(0)}] \leq \frac{\alpha(x)}{\| x \|_1} \leq \log(2d)+\E[\omega(0)],
\end{align}
where $\| \cdot \|_1$ is the $\ell^1$-norm on $\R^d$.
\end{thm}

\begin{rem}
The present paper always assumes that $\omega(0)$ has at least second moment (see hypotheses (A1) and (A2) below),
so that, for simplicity, we assume $\E[\omega(0)]<\infty$ in the above theorem.
However, it is known that the theorem is valid under lower moments.
In fact, Mourrat~\cite[Theorems~1.1]{Mou12} proved the following:
If $Z$ is the minimum of $2d$ i.i.d.\,random variables distributed as $\omega(0)$,
then for each $x \in \Z^d$,
\begin{align*}
	\E[Z]<\infty \text{ if and only if } \frac{1}{n} a(0,nx) \text{ converges a.s.}
\end{align*}
\end{rem}

\subsection{Main results}
We first introduce the following assumptions:
\begin{enumerate}
 \item[\bf (A1)]
  $\E[e^{\gamma \omega(0)}]<\infty$ for some $\gamma>0$.
 \item[\bf (A2)]
  $\E[\omega(0)^2]<\infty$.
 \item[\bf (A3)]
  The law of $\omega(0)$ has strictly positive support.
\end{enumerate}

The following two theorems are the main results of the present article.
We have the exponential concentration for the upper tail, and the Gaussian concentration for the lower tail.

\begin{thm}\label{thm:conc}
Assume (A1).
In addition, suppose that (A3) is valid if $d=2$.
Then, there exist constants $0<\Cl{1.1},\Cl{1.2}<\infty$ such that for all large $x \in \Z^d$ and for all $t \geq 0$,
\begin{align*}
	\P \bigl( a(0,x)-\E[a(0,x)] \geq t\|x\|_1^{1/2} \bigr)
	\leq \Cr{1.1}e^{-\Cr{1.2}t}.
\end{align*}
\end{thm}

\begin{thm}\label{thm:gaussian}
Assume (A2).
In addition, suppose that (A3) is valid if $d=2$.
Then, there exists a constant $0<\Cl{1.3}<\infty$ such that for all large $x \in \Z^d$
and for all $t \geq 0$,
\begin{align*}
	\P \bigl( a(0,x)-\E[a(0,x)] \leq -t\| x \|_1^{1/2} \bigr) \leq e^{-\Cr{1.3}t^2}.
\end{align*}
\end{thm}

In the case where the law of potentials has bounded and strictly positive support,
Ioffe--Velenik~\cite[Lemma~4]{IofVel12} and Sodin~\cite[Theorem~1]{Sod14} proved
the Gaussian concentration:
There exists a constant $0<c<\infty$ such that for all sufficiently large $x \in \Z^d$ and for all $t \geq 0$,
\begin{align}\label{eq:iofvel}
	\P \bigl( |a(0,x)-\E[a(0,x)]| \geq t\| x \|_1^{1/2} \bigr) \leq e^{-ct^2}.
\end{align}
On the other hand, there are few results for unbounded nonnegative potentials.
In this context, Sznitman~\cite[Theorem~2.1]{Szn96} proved exponential concentrations
for Brownian motion in a Poissonian potential.
Our work is a discrete space counterpart of this model as well as an extension of the technique to unbounded potentials.

With these observations, Theorems~\ref{thm:conc} and \ref{thm:gaussian} extend the aforementioned previous works
on concentrations to the case where the law of potentials has unbounded and (not strictly) positive support.

Finally, let us comment on the following problems:
\begin{itemize}
	\item Show the upper Gaussian concentration under the same or a weaker assumption as in Theorem~\ref{thm:conc}
		or a weaker one.
	\item Obtain the lower Gaussian concentration without additional hypothesis (A3) for $d=2$.
\end{itemize}
Intuitively, under the weighted measure appearing in $e(0,x)$,
the random walk tends to pass through sites with small potentials to reach a target point $x$ with lower cost.
If we do not assume (A3), then there will be typically big pockets where potentials are equal to zero.
In such regions the behavior of the random walk seems to be similar to the ordinary simple random walk,
so that the walker takes a lot of time to pass such a region.
This means that $H(x)$ can be large, and this increases the chance that the walker will also pass through sites
with a large potential.
In the case when the potential is unbounded, the walker can encounter very large potentials.
Avoiding such potentials makes the walker deviate a lot from its target point $x$,
and both $H(x)$ and the travel cost become large.
With these observations, for the above problems it is important to analyze the upper deviation of $H(x)$.
However, presently we do not have enough information to establish the Gaussian concentration
(see \cite[Lemmata~1 and 3]{Le13_arXiv} and \cite[Theorems~1.1 and 2.2]{Szn95}),
and seems to need an entirely different approach to do this.
We would like to address this problem in the future research.

\subsection{Earlier literature}
The simple random walk in random potentials is related to a lot of models,
e.g., first passage percolation, directed polymer in random environment, random walk in random environment, and so on.
We first mention Kesten's works~\cite{Kes86,Kes93}
for the first passage percolation on $\Z^d$.
Consider the first passage time $\tau(x,y)$, which corresponds to the travel cost $a(x,y)$, as follows:
Assign independently to each edge $e$ of $\Z^d$ a nonnegative random weight $t_e$ with a common distribution,
and define
\begin{align*}
	\tau(x,y):=\inf\biggl\{ \sum_{e \in r} t_e; r \text{ is a path on } \Z^d \text{ from } x \text{ to } y \biggr\},
	\qquad x,y \in \Z^d.
\end{align*}

The large deviation estimates for the first passage time were shown in \cite{Kes86}
(Theorem~5.2 and Theorem~5.9, the former is a counterpart of Theorem~\ref{thm:upper} above).
His idea is to construct a good self-avoiding path on $\Z^d$ approximating $\tau(0,x)$.
Then the first passage time $\tau(0,x)$ may be regarded as a sum of i.i.d.\,random variables indexed by edges of $\Z^d$,
and a standard large deviation estimate works.
Our model has to treat the random walk, and  we cannot control it around a given self-avoiding path on $\Z^d$.
Hence, Kesten's approach does not work directly.
We will overcome this problem in the proofs of Theorem~\ref{thm:upper}.

Kesten~\cite[Theorem~1]{Kes93} also obtained the exponential concentration inequality for the first passage time,
and its proof is based on a method of bounded martingale differences.
More precisely, let $\{ e_1,e_2,\dots\}$ be an enumeration of edges of $\Z^d$,
and this method represents the difference between $\tau(0,x)$ and its expectation as the sum of the martingale differences
\begin{align*}
	\Delta_i:=\E[\tau(0,x)|\mathcal{F}_i]-\E[\tau(0,x)|\mathcal{F}_{i-1}],\qquad i \geq 1,
\end{align*}
where $\mathcal{F}_0$ is the trivial $\sigma$-field and $\mathcal{F}_i$ is the $\sigma$-field generated by
$t_{e_1},\dots,t_{e_i}$.
In view of Azuma's inequality, bounds for $\Delta_i$'s govern the concentration inequality,
so that we have to estimate $\Delta_i$'s suitably as in Lemmata~\ref{lem:kes-szn1} and \ref{lem:kes-szn2},
see \cite[Theorem~3 and Section~2]{Kes93} for more details.
This method can be applied to any general random variable instead of the first passage time $\tau(0,x)$.
In our model, it is done
by estimating the size of the range of the random walk under the weighted measure appearing in $e(0,x)$.

After that, Talagrand improved Kesten's result to the Gaussian concentration by using convex hull approximation,
see \cite[Proposition~8.3]{Tal95} or \cite[Theorem~8.15]{Led05}.
We can apply this method to some Lipschitz functions.
In fact, for the proof of \eqref{eq:iofvel},
Ioffe--Velenik~\cite{IofVel12} and Sodin~\cite{Sod14} derived the Lipschitz continuity of the travel cost $a(0,x)$
in the case where the law of potentials has bounded and strictly positive support.
Its Lipschitz constant depends on the maximum and the minimum of the support of the law of the potential,
and the constant $c$ in \eqref{eq:iofvel} inherits this dependency
(see the proofs of Lemma~4 in \cite{IofVel12} and Theorem~1 in \cite{Sod14}).
Our goal is to obtain concentration inequalities for the travel cost in unbounded nonnegative potentials.
Hence, the scheme proposed by Talagrand does not seem to be easily applicable here,
and our line of approach is more in the spirit of Kesten.
As mentioned above, the same approach taken in \cite{IofVel12} and \cite{Sod14} may work
if we can control the upper deviation of $H(x)$.

As a related topic, Flury~\cite{Flu07} and Zygouras~\cite{Zyg09} studied the annealed travel cost $b(x,y)$
and its Lyapunov exponent:
\begin{align*}
	b(x,y):=-\log \E \Biggl[
	E^x \Biggl[ \exp \Biggl\{ -\sum_{k=0}^{H(y)-1} \omega (S_k) \Biggr\} \1{\{ H(y)<\infty \}} \Biggr] \Biggr]
\end{align*}
and
\begin{align*}
	\lim_{n \to \infty} \frac{1}{n}b(0,nx)
	= \inf_{n \geq 1} \frac{1}{n}\E[b(0,nx)]
	= \beta(x).
\end{align*}
In particular, Zygouras gave a sufficient condition for the equality of the quenched and the annealed Lyapunov exponents
in $d \geq 4$.

Flury~\cite{Flu08} also investigated the quenched and the annealed partition functions: For $\beta \geq 0$,
\begin{align*}
	Z_{\beta,n}^{\text{qu}}
	:=E^{\text{RW}} \biggl[ \exp\biggl\{ -\beta \sum_{k=0}^n \omega(X_k) \biggr\} \biggr],\quad
	Z_{\beta,n}^{\text{an}}
	:=\E \biggl[ E^{\text{RW}} \biggl[ \exp\biggl\{ -\sum_{k=0}^n \omega(X_k) \biggr\} \biggr] \biggr],
\end{align*}
respectively.
Here $(X_k)_{k=0}^\infty$ is the random walk on $\Z^d$ starting at $0$ with constant drift in the first coordinate direction
and $E^{\text{RW}}$ is the expectation with respect to its law.
Theorem~C of \cite{Flu08} shows that if the drift is positive and $\beta$ is sufficiently small, then
\begin{align*}
	-\frac{1}{n}Z_{\beta,n}^{\text{qu}} \sim -\frac{1}{n}Z_{\beta,n}^{\text{an}}
	\text{ as } n \to \infty.
\end{align*}
Moreover, the boundedness of potentials enables us to estimate its speed of the convergence.
To do this, Flury considered the quenched partition function conditioned on the random walk trajectories
whose range has a proper size (see page~1535 of \cite{Flu08}).
This derives the Lipschitz continuity, and Talagrand's concentration inequality works to compare
the quenched and annealed partition functions.

The large deviation estimates and the concentration inequalities
have also been studied for the directed polymer in random environment.
In this model, we consider i.i.d.\,random variables $(\omega(x,k))_{x \in \Z^d, k \geq 0}$ as space-time potentials,
and the simple random walk $(S_k)_{k=0}^\infty$ on $\Z^d$ as a random walk.
A main object of this model is the quenched partition function $Z_{\beta,n}^{\text{qu}}$ as above.
Refer the reader to articles and lectures \cite{ComShiYos03,ComShiYos04,ComYos06,dH09} for more details.
In this model, due to superadditive arguments, it is known that the upper tail of the large deviations is exponential
for the quenched partition function.
Ben-Ari~\cite{BA09} thus studied the explicit order of the lower tail large deviations.
Corollary~1 of \cite{BA09} provides a necessary and sufficient condition
for order of the lower tail large deviations to be comparable to the exponential tail of the potential.
There are few results for the explicit rate function of this model.
In this context, Georgiou--Sepp\"{a}l\"{a}inen~\cite{GeoSep13} investigated the upper tail large deviations
for a version of the directed polymer in random environment, and
they explicitly gave the form of the upper tail large deviation rate function
for the $(1+1)$-dimensional log-gamma distributed potentials.
We need additional research to decide the lower tail large deviation rate function explicitly,
see Remark~2.4 of \cite{GeoSep13}.
In terms of the concentration inequalities, we can find recent results in \cite{CarHu02,LiuWat09,Mor10,Wat12}.
Talagrand's method for the Lipschitz function is basically used in \cite{CarHu02,Mor10,Wat12}.
Under an exponential tail assumption for the potential, Liu--Watbled~\cite{LiuWat09} and Watbled~\cite{Wat12} basically follow
the martingale method with a generalization of Hoeffding's inequality.
This works very well for the above partition function induced by the simple random walk up to step $n$.
Since our model is governed by the random walk before hitting a site $x$,
one has to derive a good upper tail estimate for $H(x)$ under the weighted measure.
Thus, we cannot directly apply their method as mentioned at the end of the above subsection.

For the standard random walk in random environment, we consider the travel cost
in the case where the potential is always equal to a positive constant,
i.e., the Laplace transform of the hitting time for the random walk in random environment.
It governs the rate function for the large deviation principle for the position of the random walk,
see \cite{Kub12b,RASepYil13,Zer98b}.
The rate function is given by the Legendre transform of the Lyapunov exponents, and 
we expect that our analysis is useful for such a problem.

\subsection{Organization of the paper}
Let us now describe how the present article is organized.
In Section~\ref{sect:uld}, we prepare a large deviation inequality for the upper tail
to prove Theorem~\ref{thm:conc} (see Theorem~\ref{thm:upper}).
To this end, we introduce travel costs restricted to suitable blocks and construct appropriate i.i.d.\,random variables.
A large deviation argument then derives the directional version of Theorem~\ref{thm:upper} (see Lemma~\ref{lem:udirect}).
To extend Lemma~\ref{lem:udirect} uniformly for all directions,
we need a modification of the so-called maximal lemma in \cite[Lemma~7]{Zer98a} under (A1),
which is Lemma~\ref{lem:maximal}.

The goal of Section~\ref{sect:conc} is to prove Theorem~\ref{thm:conc}.
Our main tool here is the martingale method as in \cite[Section~2]{Kes93} or \cite[Section~2]{Szn96}.
To apply this, we have to control an upper bound on how much $a(0,x,\omega)$ may change when $\omega$ is changed.
The rank-one perturbation formula, which is proved by \cite[Lemma~12]{Zer98a}, observes effect
of a change at a single site (see Lemma~\ref{lem:rankone}).
Furthermore, by using Theorem~\ref{thm:upper},
one gives a bound on how much $a(0,x,\omega)$ may change when $\omega$ is changed at finitely many sites.

In Section~\ref{sect:gaussian}, we will show Theorem~\ref{thm:gaussian}.
A main tool is the so-called entropy method.
The entropy has been well-studied in statistical mechanics, and is a fundamental tool
for large deviation principles, see \cite{Ell12_book} for example.
Recently,  Ledoux~\cite{Led05} and Boucheron--Lugosi--Massart~\cite{BouLugMas13_book}
developed concentrations of measure phenomenon by using the entropy method (based on certain
logarithmic Sobolev inequalities).
For the proof of Theorem~\ref{thm:gaussian} we basically follow their approach as well as \cite{DamKub14_arXiv}.
The hypothesis of the theorem guarantees that under the weighted measure appearing in $e(0,x)$,
the size of the range of the random walk grows at most linearly
in the $\ell^1$-norm of a target point, see Proposition~\ref{prop:mean_la}.
This proposition has already mentioned in \cite[Lemma~2]{Le13_arXiv} with only a comment on the proof.
It plays a key role in the entropy method and we give its detailed proof.

We close this section with some general notation.
Write $\| \cdot \|_1$, $\| \cdot \|_2$ and $\| \cdot \|_\infty$
for the $\ell^1$, $\ell^2$ and $\ell^\infty$-norms on $\R^d$, respectively.
Moreover, in each section, we use $C_i$, $i=1,2,\dots$, to denote constants with $0<C_i<\infty$,
whose precise values are not important to us, and whose value may be different in different sections.

\section{A large deviation inequality for the upper tail}\label{sect:uld}
In this section, let us prove the following large deviation inequality for the upper tail.
A part of the proof of Theorem~\ref{thm:conc} relies on it.

\begin{thm}\label{thm:upper}
Assume (A1).
Then, for all $\epsilon>0$ there exist constants $\Cl{2.1},\Cl{2.2}$ such that
for all $x \in \Z^d$,
\begin{align}\label{eq:upper}
	\P (a(0,x)-\alpha(x) \geq \epsilon \| x \|_1 )
	\leq \Cr{2.1}e^{-\Cr{2.2} \| x \|_1}.
\end{align}
\end{thm}

To do this, we prepare some notation and lemmata.
Write $T_V$ for the exit time of $(S_k)_{k=0}^\infty$ from a subset $V$ of $\R^d$, i.e.,
\begin{align*}
 T_V:=\inf \{ k \geq 0; S_k \not\in V \}.
\end{align*}
Then, we consider the travel cost from $x$ to $y$ restricted to the random walk before exiting $V$ as follows:
\begin{align*}
	a_V(x,y):=-\log e_V(x,y),
\end{align*}
where
\begin{align*}
	e_V(x,y):=E^x \Biggl[ \exp \Biggl\{ -\sum_{k=0}^{H(y)-1} \omega (S_k) \Biggr\}
	\1{\{ H(y)<T_V \}} \Biggr].
\end{align*}

For any $\xi \in \R^d \setminus \{ 0 \}$, let $\mathcal{R}_\xi$ be a rotation of $\R^d$ with
\begin{align*}
	\mathcal{R}_\xi (e_1)=\frac{\xi}{\| \xi \|_2},
\end{align*}
where $e_1$ is the first coordinate direction of $\R^d$.
Then, for $n>m \geq 0$ and $N \geq 0$ we consider a block
\begin{align*}
	\mathcal{P}_{m,n}^N(\xi )
	:= \mathcal{R}_\xi \bigl( [m\| \xi \|_2 -N,n\| \xi \|_2 +N] \times [-N,N]^{d-1} \bigr) \cap \Z^d.
\end{align*}
If $m\xi,n\xi \in \Z^d$, then write $a_{m,n}(\xi):=a(m\xi,n\xi)$
and $T_{m,n}^N(\xi):=T_{\mathcal{P}_{m.n}^N(\xi)}$,
and let $a_{m,n}^N(\xi)$ be the travel cost from $m\xi$ to $n\xi$
restricted to the random walk before exiting $\mathcal{P}_{m,n}^N(\xi)$,
i.e.,
\begin{align*}
	a_{m,n}^N(\xi ):=a_{\mathcal{P}_{m,n}^N(\xi)}(m\xi,n\xi).
\end{align*}
It is clear that $a_{m,n}(\xi) \leq a_{m,n}^N(\xi)$
and the sequence $a_{m,n}^N(\xi)$, $n>m \geq 0$, is subadditive.

\begin{lem}\label{lem:decr}
For all $\xi \in \Z^d \setminus \{ 0 \}$, $a_{m,n}^N(\xi)$ converges decreasingly
to $a_{m,n}(\xi)$ as $N \to \infty$.
\end{lem}
\begin{proof}
Since the event $\{ H(n\xi)<T_{m,n}^N(\xi) \}$ is increasing in $N$,
$\1{\{ H(n\xi)<T_{m,n}^N(\xi) \}}$ converges increasingly to $\1{\{ H(n\xi)<\infty \}}$ as $N \to \infty$.
Therefore, the lemma follows from the monotone convergence theorem.
\end{proof}

\begin{lem}\label{lem:udirect}
Assume (A1).
Then, for all $\epsilon >0$ and $\xi \in \Z^d \setminus \{ 0 \}$
there exist constants $\Cl{2.3}=\Cr{2.3}(\xi),\Cl{2.4}=\Cr{2.4}(\xi)$ such that
\begin{align}\label{eq:udirect}
 \P (a_{0,n}(\xi ) \geq n(\alpha (\xi )+\epsilon \| \xi \|_1 ))
 \leq \Cr{2.3}e^{-\Cr{2.4}n\| \xi \|_1},\qquad n \geq 0.
\end{align}
\end{lem}
\begin{proof}
Let $\epsilon >0$ and $\xi \in \Z^d \setminus \{ 0 \}$ be given.
By \eqref{eq:lyapunov} we can choose $\nu=\nu(\xi) \in \N$ such that
\begin{align*}
 \E[a_{0,\nu}(\xi )]\leq \nu \Bigl( \alpha (\xi )+\frac{\epsilon}{5}\| \xi \|_1 \Bigr).
\end{align*}
Note that if $N' \in \N$ is large enough, then we can pick a path $r'$ on $\Z^d$ from $0$ to $\nu\xi$
which is contained in $\mathcal{P}_{0,\nu}^{N'}(\xi)$.
This implies that for all $N \geq N'$,
\begin{align*}
 a_{0,\nu}^N(\xi) \leq \sum_{z \in r'} (\omega(z) +\log(2d) ),
\end{align*}
and Lemma~\ref{lem:decr} and Lebesgue's dominated convergence theorem show
\begin{align*}
 \lim_{N \to \infty} \E[a_{0,\nu}^N(\xi )] =\E[a_{0,\nu}(\xi )].
\end{align*}
With these observations, there is an $N=N(\xi)$ such that
\begin{align}\label{eq:N_choice}
 \E[a_{0,\nu}^N(\xi )]
 \leq \nu \Bigl( \alpha (\xi )+\frac{2\epsilon}{5} \| \xi \|_1 \Bigr).
\end{align}

We now prove \eqref{eq:udirect} with $n=\pi\nu$ and $\pi \in \N$.
Since $\bigcup_{j=0}^{\pi-1} \mathcal{P}_{j\nu,(j+1)\nu}^N(\xi)=\mathcal{P}_{0,\pi\nu}^N(\xi)$
for $\pi \in \N$, the subadditivity shows
\begin{align}\label{eq:blk_sadd}
 a_{0,\nu \pi}^N(\xi ) \leq \sum_{j=0}^{\pi-1} a_{j\nu,(j+1)\nu}^N(\xi ),\qquad \pi \in \N.
\end{align}
The sequence $a_{j\nu,(j+1)\nu}^N(\xi)$, $j \geq 0$, has the following properties:
\begin{itemize}
 \item
  For $L:=\lceil 2N\nu^{-1}  \rceil +2$, $a_{j\nu,(j+1)\nu}^N(\xi)$'s are $L$-dependent,
  i.e., any two sequences $a_{j\nu,(j+1)\nu}^N(\xi)$, $j \in \Lambda$,
  and $a_{j\nu,(j+1)\nu}^N(\xi)$, $j \in \Lambda'$, are independent
  whenever $\Lambda,\,\Lambda' \subset \N_0$ satisfy that $|j-j'|>L$ for all $j \in \Lambda$, $j' \in \Lambda'$.
	\item
		All $a_{j\nu,(j+1)\nu}^N(\xi)$'s have the same distribution as $a_{0,\nu}^N(\xi)$.
\end{itemize}
These, together with \eqref{eq:blk_sadd} and Chebyshev's inequality, imply that for $\gamma' \in (0,\gamma/2)$,
\begin{align*}
 &\P \biggl( a_{0,\pi\nu}^N(\xi )
  \geq \pi \nu \Bigl( \alpha (\xi )+\frac{3\epsilon}{5}\| \xi \|_1 \Bigr) \biggr)\\
 &\leq \sum_{i=0}^{L-1} \P \Biggl(
       \sum_{\substack{0 \leq j \leq \pi-1\\ j \bmod L=i}}
       a_{j\nu, (j+1)\nu}^N(\xi )
       \geq \frac{1}{L} \pi \nu \Bigl( \alpha (\xi )+\frac{3\epsilon}{5}\| \xi \|_1 \Bigr) \Biggr)\\
 &\leq L \exp \Bigl\{ -\frac{\gamma'}{L} \pi \nu \Bigl( \alpha (\xi )
       +\frac{3\epsilon}{5}\| \xi \|_1 \Bigr) \Bigr\}
       \E \bigl[ \exp \{ \gamma' a_{0,\nu}^N(\xi ) \} \bigr]^{\pi/L}.
\end{align*}

We now estimate the expectation above.
By \eqref{eq:N_choice}, one has
\begin{align*}
	\E \bigl[ \exp \{ \gamma' a_{0,\nu}^N(\xi ) \} \bigr]
	&\leq \E \bigl[ \exp \big\{ \gamma' (a_{0,\nu}^N(\xi )-\E[a_{0,\nu}^N(\xi )] ) \bigr\} \bigr]
		\exp \Bigl\{ \gamma' \nu \Bigl( \alpha (\xi )+\frac{2\epsilon}{5}\| \xi \|_1 \Bigr) \Bigr\}.
\end{align*}
Since $e^t \leq 1+te^t$ for $t \in \R$,
this is smaller than
\begin{align}\label{eq:exp_bound}
\begin{split}
	&\Bigl( 1+\gamma' \E \Bigl[ \bigl( a_{0,\nu}^N(\xi )-\E[a_{0,\nu}^N(\xi )] \bigr)
		\exp \bigl\{ \gamma' \bigl( a_{0,\nu}^N(\xi )-\E[a_{0,\nu}^N(\xi )] \bigr) \bigr\} \Bigr] \Bigr)\\
	&\times \exp \Bigl\{ \gamma' \nu \Bigl( \alpha (\xi )+\frac{2\epsilon}{5}\| \xi \|_1 \Bigr) \Bigr\}.
\end{split}
\end{align}
Due to (A1), we can use Lebesgue's dominated convergence theorem to get
\begin{align*}
 \lim_{\gamma' \searrow 0} \E \Bigl[ \bigl( a_{0,\nu}^N(\xi )-\E[a_{0,\nu}^N(\xi )] \bigr)
 \exp \bigl\{ \gamma' \bigl( a_{0,\nu}^N(\xi )-\E[a_{0,\nu}^N(\xi )] \bigr) \bigl\} \Bigr]
 =0.
\end{align*}
This enables us to take $\gamma'=\gamma'(\xi) \in (0,\gamma/2)$ such that
\begin{align*}
 \E \Bigl[ \bigl( a_{0,\nu}^N(\xi )-\E[a_{0,\nu}^N(\xi )] \bigr)
 \exp \bigl\{ \gamma' \bigl( a_{0,\nu}^N(\xi )-\E[a_{0,\nu}^N(\xi )] \bigr) \bigl\} \Bigr]
 \leq \frac{\epsilon}{10}\nu\| \xi \|_1.
\end{align*}
Combining this and \eqref{eq:exp_bound}, one has
\begin{align*}
 \E [ \exp \{ \gamma' a_{0,\nu}^N(\xi ) \} ]
 \leq \Bigl( 1+\frac{\gamma' \epsilon}{10} \nu \| \xi \|_1 \Bigr)
      \exp \Bigl\{ \gamma' \nu \Bigl( \alpha (\xi )+\frac{2\epsilon}{5} \| \xi \|_1 \Bigr) \Bigr\}.
\end{align*}

With these observations, we have for all $\pi \in \N$,
\begin{align*}
 \P \biggl( a_{0,\pi\nu}^N(\xi ) \geq
 \pi \nu \Bigl( \alpha (\xi )+\frac{3\epsilon}{5}\| \xi \|_1 \Bigr) \biggr)
 &\leq L \biggl( \Bigl( 1+\frac{\gamma' \epsilon}{10} \nu \| \xi \|_1 \Bigr)
       \exp \Bigl\{ -\frac{\gamma' \epsilon}{5} \nu \| \xi \|_1 \Bigl\} \biggr)^{\pi/L}\\
 &\leq L\exp \Bigl\{ -\frac{\gamma' \epsilon}{10L} \pi \nu \| \xi \|_1 \Bigr\},
\end{align*}
and hence \eqref{eq:udirect} immediately follows from
the fact that $a_{0,\pi\nu}(\xi) \leq a_{0,\pi\nu}^N(\xi)$ in this case.

If $n=\pi \nu +\tau$ with $0<\tau <\nu$, then the subadditivity yields
\begin{align*}
	a_{0,n}(\xi )
	\leq a_{0,n}^N(\xi )
	\leq a_{0,\pi \nu}^N(\xi )+a_{\pi \nu,n}^N(\xi ).
\end{align*}
We use Chebyshev's inequality to obtain
\begin{align*}
	\P \Bigl( a_{\pi \nu,n}^N(\xi ) \geq \frac{\epsilon}{5}n\| \xi \|_1 \Bigr)
	\leq \exp \Bigl\{ -\frac{\gamma \epsilon}{5}n \| \xi \|_1 \Bigr\}
	\E \bigl[ \exp \{ \gamma (\omega (0)+\log (2d)) \} \bigr]^{\tau \| \xi \|_1}.
\end{align*}
It follows that
\begin{align*}
	&\P \biggl( a_{0,n}(\xi ) \geq n \Bigl( \alpha (\xi )+\frac{4\epsilon}{5} \| \xi \|_1 \Bigr) \biggr)\\
	&\leq \P \biggl( a_{0,\pi \nu}^N(\xi ) \geq
		n \Bigl( \alpha (\xi )+\frac{3\epsilon}{5}\| \xi \|_1 \Bigr) \biggr)
		+\P \Bigl( a_{\pi \nu,n}^N(\xi ) \geq \frac{\epsilon}{5}n\| \xi \|_1 \Bigr)\\
	&\leq L\exp \Bigl\{ -\frac{\gamma' \epsilon}{10L}(n-\tau )\| \xi \|_1 \Bigr\}
		+\exp \Bigl\{ -\frac{\gamma \epsilon}{5} n\| \xi \|_1 \Bigr\}
		\E \bigl[ \exp \{ \gamma (\omega (0)+\log (2d)) \} \bigr]^{\tau \| \xi \|_1}.
\end{align*}
Accordingly, we complete the proof of \eqref{eq:udirect} in all cases
by choosing $\Cr{2.3}$ and $\Cr{2.4}$ suitably.
\end{proof}

To extend Lemma~\ref{lem:udirect} uniformly for all directions,
the next lemma is useful, which is a modification of the maximal lemma proved by \cite[Lemma~7]{Zer98a}.

\begin{lem}\label{lem:maximal}
If (A1) holds, then there exist constants $\Cl{2.5}$, $\Cl{2.6}$ such that
for all $\eta >0$ and $x \in \Z^d$,
\begin{align}\label{eq:maximal}
	\P \bigl( \sup \{ d(x,y); y \in \Z^d,\,\| x-y \|_1 <\eta \| x \|_1 \} \geq \Cr{2.5} \eta \| x \|_1 \bigr)
	\leq e^{-\Cr{2.6} \eta \| x \|_1},
\end{align}
where $d(x,y):=a(x,y) \vee a(y,x)$.
\end{lem}
\begin{proof}
For $x,y \in \Z^d$ with $\| x-y \|_1<\eta \| x \|_1$,
let $r$ be the shortest path on $\Z^d$ from $x$ to $y$.
Use Chebyshev's inequality to obtain that for any $c>0$,
\begin{align*}
 \P (d(x,y) \geq c\eta \| x \|_1 )
 &\leq \P \biggl( \sum_{z \in r}(\omega(z)+\log(2d) ) \geq c\eta \| x \|_1 \biggr)\\
 &\leq e^{-c\gamma \eta \| x \|_1}
       \E \bigl[ \exp\{ \gamma (\omega(0) +\log(2d) )\} \bigr]^{\eta \| x \|_1+1}.
\end{align*}
The union bound shows that
\begin{align*}
 &\P \bigl( \sup \{ d(x,y); y \in \Z^d,\,\| x-y \|_1 <\eta \| x \|_1 \} \geq c \epsilon \| x \|_1 \bigr)\\
 &\leq (2\eta \| x \|_1 +1)^d
       e^{-c\gamma \eta \| x \|_1}
       \E \bigl[ \exp\{ \gamma (\omega(0) +\log(2d) )\} \bigr]^{\eta \| x \|_1},
\end{align*}
which proves \eqref{eq:maximal} by choosing $c$ large enough.
\end{proof}

Now we are in a position to prove Theorem~\ref{thm:upper}.

\begin{proof}[\bf Proof of Theorem~\ref{thm:upper}]
Since $S^{d-1}:=\{ y \in \R^d;\| y \|_1=1 \}$ is compact,
we can take a subset $\{ v_1,\dots,v_l \}$ of $S^{d-1} \cap \Q^d$ satisfying that
for any $x \in \R^d \setminus \{ 0 \}$,
there is $i(x) \in [1,l]$ such that
$\| x/\| x \|_1 -v_{i(x)} \|_1<\epsilon/(6C_3)$
and $|\alpha (x/\| x \|_1)-\alpha (v_{i(x)})|<\epsilon/3$.
For each $i \in [1,l]$, choose $M_i \in \N$ with $M_i v_i \in \Z^d$.
Let $x \in \Z^d$ be given and set
\begin{align*}
	n(x):=\biggl\lfloor \frac{\| x \|_1}{M_{i(x)}} \biggr\rfloor,\qquad
	\xi(x) :=M_{i(x)}v_{i(x)}.
\end{align*}
Lemma~\ref{lem:udirect} implies
\begin{align*}
	&\P \biggl( a_{0,n(x)}(\xi (x)) \geq n(x) \Bigl( \alpha (\xi (x))+\frac{\epsilon}{3}\| \xi(x) \|_1 \Bigr) \biggr)\\
	&\leq \Cr{2.3}(\xi (x)) e^{-\Cr{2.4}(\xi (x))n(x)\| \xi(x) \|_1}.
\end{align*}
Let $\Cl{2.7}:=\max_{1 \leq i \leq l} \Cr{2.3}(M_iv_i)$ and $\Cl{2.8}:=\min_{1 \leq i \leq l} \Cr{2.4}(M_iv_i)$.
Since $\xi(x)$ is of the form $M_iv_i$, $\Cr{2.3}(\xi (x)) \leq \Cr{2.7}$ and $\Cr{2.4}(\xi (x)) \geq \Cr{2.8}$
uniformly in $x$.
On the other hand, for all large $x \in \Z^d$,
\begin{align*}
	\| x-n(x)\xi(x) \|_1
	\leq \| x \|_1 \biggl\| \frac{x}{\| x \|_1}-v_{i(x)} \biggr\|_1 +\max_{1 \leq i \leq l}M_i
	< \frac{\epsilon}{3\Cr{2.5}} \| x \|_1,
\end{align*}
so that Lemma~\ref{lem:maximal} implies
\begin{align*}
	&\P (a(0,x)-\alpha (x) \geq \epsilon \| x \|_1 )\\
	&\leq \P \biggl( a_{0,n(x)}(\xi (x)) \geq n(x) \Bigl( \alpha (\xi (x))+\frac{\epsilon}{3} \| \xi(x) \|_1 \Bigr) \biggr)
		+\P \Bigl( d(x,n(x)\xi(x) ) \geq \frac{\epsilon}{3} \| x \|_1 \Bigr)\\
	&\leq \Cr{2.7} e^{-\Cr{2.8}n(x)\| \xi(x) \|_1} +\exp \biggl\{ -\frac{\Cr{2.6}\epsilon}{3\Cr{2.5}} \| x \|_1 \biggr\}.
\end{align*}
Notice that
\begin{align*}
 n(x)\| \xi(x) \|_1
 \geq \biggl( \frac{\| x \|_1}{M_{i(x)}}-1 \biggr)M_{i(x)}
 \geq \| x \|_1-\max_{1 \leq i \leq l}M_i,
\end{align*}
and the theorem immediately follows.
\end{proof}

\section{The concentration for the upper tail}\label{sect:conc}
To show Theorem~\ref{thm:conc}, we basically follow from the strategy taken
in \cite[Section~2]{Kes93} or \cite[Section~2]{Szn96}.
In Subsection~\ref{subsect:compare}, let us observe how much $a(0,x,\omega)$ may change when $\omega$ is changed.
Observations of Subsection~\ref{subsect:compare} allow us to apply the martingale argument
taken in \cite[Section~2]{Kes93}.
The proof of Theorem~\ref{thm:conc} is done in Subsection~\ref{subsect:pf_conc}.

\subsection{Comparison of travel costs}\label{subsect:compare}
For $M \in \N$ we consider the boxes $B(q):=Mq+[0,M)^d$, $q \in \Z^d$.
These boxes form a partition of $\R^d$, and let $q(x)$ be the index such that $x \in B(q(x))$.
Fix $\kappa >0$ satisfying $\P(\omega(0) \geq \kappa)>0$.
Given $\omega \in \Omega$,
$B(q)$ is said to be occupied if there is a site $x \in B(q)$ such that $\omega (x) \geq \kappa$.
Let $\tau_0=0$ and define for $i \geq 0$,
\begin{align*}
	&\rho_{i+1}:=\inf \{ k>\tau_i; S_k \text{ reaches an occupied box} \},\\
	&\tau_{i+1}:=\inf \{ k>\rho_{i+1}; S_k \not\in B(q(S_{\rho_{i+1}})) \}.
\end{align*}
Furthermore, an $\ell^1$-lattice animal is a finite $\ell^1$-connected set of $\Z^d$,
when the adjacency relation of two sites $v_1,v_2 \in \Z^d$ is defined as $\|v_1-v_2\|_1=1$.
\label{page:def_animal}

\begin{lem}\label{lem:restrict}
There exist constants $\Cl{3.1}$, $\Cl{3.2}$, $\Cl{3.3}$ such that if $m$ is large enough,
then, on an event $A_m$ with $\P(A_m^c) \leq \Cr{3.1}e^{-\Cr{3.2}m}$, we have
\begin{align*}
	e(0,x) \leq e^{-\Cr{3.3}m}
\end{align*}
for all $x \in \Z^d$ with $\|x\|_1=m$.
\end{lem}
\begin{proof}
Let $A'_l$ be the event that for any $\ell^1$-lattice animal $\Gamma$ containing $0$ with $\#\Gamma=l$,
\begin{align*}
	\sum_{q \in \Gamma} \1{\{B(q) \text{ is occupied}\}} \geq \frac{l}{2}.
\end{align*}
Notice that the number of $\ell^1$-lattice animals of size $l$ is bounded by $(2^d)^{2l}=4^{dl}$
(see Lemma~1 of \cite{CoxGanGriKes93}).
Since $\P(B(0) \text{ is occupied}) \to 1$ as $M \to \infty$,
for $M$ large enough, standard exponential estimates on the binomial distribution show
\begin{align}\label{eq:bin_conc}
	\P((A'_l)^c)
	\leq 4^{dl} \P \biggl( \sum_{q \in \Gamma} \1{\{B(q) \text{ is occupied}\}}<\frac{l}{2} \biggr)
	\leq e^{-\Cr{3.4}l}
\end{align}
for some constant $\Cl{3.4}$.
Let $A_m$ be the event that for any $\ell^1$-lattice animal $\Gamma$ on $\Z^d$ containing $0$ of the size bigger than $m/M$,
\begin{align*}
	\sum_{q \in \Gamma} \1{\{ B(q) \text{ is occupied}\}} \geq \frac{\#\Gamma}{2}.
\end{align*}
By \eqref{eq:bin_conc}, there exist some constants $\Cr{3.1}$ and $\Cr{3.2}$ such that
\begin{align*}
	\P(A_m^c) \leq \sum_{l=1}^L \P((A'_l)^c) \leq \Cr{3.1}e^{-\Cr{3.2}m},
\end{align*}
where $L:=\lfloor m/M \rfloor$.

For $x \in \Z^d$ with $\|x\|_1=m$, the random walk from $0$ to $x$ must pass through
at least $L$ boxes $B(q)$.
Hence, on the event $A_m$,
\begin{align*}
	e(0,x)
	&\leq E^0 \Biggl[ \prod_{i=1}^L \exp \Biggl\{ -\sum_{k=\rho_i}^{\tau_i-1}\omega (S_k) \Biggr\} \Biggr].
\end{align*}
If $z \in \Z^d$ is in an occupied box, then
\begin{align*}
	E^z \Biggl[ \exp \Biggl\{ -\sum_{k=0}^{\tau_1-1}\omega (S_k) \Biggr\} \Biggr]
	&\leq E^z \Biggl[ \exp \Biggl\{ -\sum_{k=0}^{\tau_1-1}\kappa \1{\{ \omega (S_k) \geq \kappa \}} \Biggr\}
		\Biggr]\\
	&\leq 1-(1-e^{-\kappa}) P^z \Bigl( \max_{0 \leq k<\tau_1} \omega(S_k) \geq \kappa \Bigr)\\
	&\leq 1-(1-e^{-\kappa}) \Bigl( \frac{1}{2d} \Bigr)^M.
\end{align*}
We thus use the strong Markov property to obtain that on the event $A_m$,
there exists a constant $\Cr{3.3}$ such that
\begin{align*}
	e(0,x)
	\leq \biggl\{ 1-(1-e^{-\kappa}) \Bigl( \frac{1}{2d} \Bigr)^M \biggr\}^L
	\leq e^{-\Cr{3.3}m},
\end{align*}
which completes the proof.
\end{proof}

\begin{prop}\label{prop:compare}
Assume (A1).
Then, there exist constants $\Cl{3.5}$, $\Cl{3.6}$, $\Cl{3.7}$ such that for $x \in \Z^d \setminus \{0\}$
and for $V=[-\Cr{3.5}\|x\|_1,\Cr{3.5}\|x\|_1]^d$,
\begin{align}\label{eq:compare}
	\P \bigl( a(0,x)<a_V(0,x)-\log 2 \bigr) \leq \Cr{3.6}e^{-\Cr{3.7}\|x\|_1}
\end{align}
and
\begin{align}\label{eq:com_exp}
	\sup_{x \in \Z^d} \bigl| \E[a_V(0,x)]-\E[a(0,x)] \bigr|<\infty.
\end{align}
\end{prop}
\begin{proof}
Put $m(c)=\lceil c\|x\|_1 \rceil$ and $V(c):=[-m(c),m(c)]^d$ for $c>0$.
From Theorem~\ref{thm:upper} and Lemma~\ref{lem:restrict},
the left side of \eqref{eq:compare} is bounded from above by
\begin{align}\label{eq:zero}
\begin{split}
	&\Cr{3.1} e^{-\Cr{3.2}c\|x\|_1}\\
	&+\P \bigl( \{ a(0,x)<a_{V(c)}(0,x)-\log 2, \,a(0,x)-\alpha(x) \leq \|x\|_1 \} \cap A_{m(c)} \bigr),
\end{split}
\end{align}
where $A_{m(c)}$ is the event as in Lemma~\ref{lem:restrict}.
Our task is now to estimate the last probability.
Thanks to Lemma~\ref{lem:restrict}, there exists a constant $\Cl{3.8}$ such that on the event $A_{m(c)}$,
\begin{align*}
	&E^0 \Biggl[ \exp \Biggl\{ -\sum_{k=0}^{H(x)-1}\omega(S_k) \Biggr\} \1{\{ T_{V(c)}<H(x)<\infty \}} \Biggr]\\
	&\leq E^0 \Biggl[ \exp \Biggl\{ -\sum_{k=0}^{T_{V(c)}-1}\omega(S_k) \Biggr\} \1{\{ T_{V(c)}<\infty \}} \Biggr]
	\leq e^{-\Cr{3.8}m(c)}.
\end{align*}
It follows that on the event $\{ a(0,x)-\alpha(x) \leq \|x\|_1 \} \cap A_{m(c)}$,
\begin{align}	\label{eq:again}
	\alpha (x)+\|x\|_1
	\geq a(0,x)
	\geq -\log \bigl( e^{-\Cr{3.8}m(c)}+e^{-a_{V(c)}(0,x)} \bigr).
\end{align}
Choose $c$ large enough.
If $a_{V(c)}(0,x)>\Cr{3.8} m(c)$ holds, then \eqref{eq:lbd} and \eqref{eq:again} derive
\begin{align*}
	\alpha (x)+\|x\|_1
	\geq -\log 2+\Cr{3.8}m(c)
	> \alpha (x)+\|x\|_1.
\end{align*}
This is a contradiction,
and hence $a_{V(c)}(0,x) \leq \Cr{3.8} m(c)$.
This, combined with the second inequality in \eqref{eq:again}, proves that 
on the event $\{ a(0,x)-\alpha(x) \leq \|x\|_1 \} \cap A_{m(c)}$,
\begin{align*}
	a(0,x) \geq -\log \bigl( e^{-a_{V(c)}(0,x)}+e^{-a_{V(c)}(0,x)} \bigr)=a_{V(c)}(0,x)-\log 2.
\end{align*}
Therefore, the last probability in \eqref{eq:zero} is equal to zero,
and we finished the proof of \eqref{eq:compare}.
On the other hand, \eqref{eq:com_exp} is an immediate consequence of \eqref{eq:compare}.
\end{proof}

Fix an arbitrary $x \in \Z^d \setminus \{ 0 \}$, and define
\begin{align*}
	\hat{\omega}(\cdot):=\omega (\cdot) \wedge \frac{4d}{\gamma} \log \|x\|_1
\end{align*}
with $\gamma$ as in assumption (A1).
For $V$ as in the above proposition, we consider
\begin{align*}
	\tilde{a}(0,x):=a_V(0,x),\qquad \hat{a}(0,x):=\tilde{a}(0,x,\hat{\omega}).
\end{align*}
Then, the next proposition, which is our goal in this subsection, observes that these travel costs are comparable.

\begin{prop}\label{prop:trunc}
Assume (A1).
In addition, suppose that (A3) is valid if $d=2$.
Then, there exists a constant $\Cl{3.9}$ such that
\begin{align}\label{eq:ptrunc}
	\P \bigl( \tilde{a}(0,x)-\hat{a}(0,x) \geq u \bigr) \leq \Cr{3.9}e^{-(\gamma /2)u},\qquad u \geq 0,
\end{align}
and
\begin{align} \label{eq:pexpect}
	\sup_{x \in \Z^d} \bigl| \E[\tilde{a}(0,x)]-\E[\hat{a}(0,x)] \bigr| <\infty.
\end{align}
\end{prop}

For the convenience, we here refer to the following rank-one perturbation formula obtained in \cite[Lemma~{12}]{Zer98a}.
This gives an upper bound on how much $a(0,x,\omega)$ may change when $\omega$ is changed  at a single site.
We omit the proof and refer the reader to that of \cite[Lemma~{12}]{Zer98a}.

\begin{lem}[Zerner]\label{lem:rankone}
Let $y \in \Z^d$ and $\omega,\sigma \in \Omega$ such that
$\omega(z)=\sigma(z)$ for $z \not=y$ and $\omega(y) \leq \sigma(y)$.
Then, for $V=\Z^d$ or $[-\Cr{3.5}\|x\|_1,\Cr{3.5}\|x\|_1]^d$, $a_V(0,x,\sigma)-a_V(0,x,\omega)$ is nonnegative,
and is bounded from above by the minimum of
\begin{align*}
	-\log Q_\omega^{0,x}(H(x) \leq H(y))
\end{align*}
and
\begin{align*}
	\sigma(y)-\omega(y)+\frac{1}{1-\min\{ e^{-\omega(y)},P^0(H_2(0)<\infty) \}},
\end{align*}
where $H_2(0)$ is the time of the second visit of $0$ for the random walk,
and $Q_\omega^{0,x}$ is the probability measure such that
\begin{align*}
	\frac{dQ_\omega^{0,x}}{dP^0}
	= e(0,x,\omega)^{-1} \exp \Biggl\{ -\sum_{k=0}^{H(x)-1} \omega(S_k) \Biggr\} \1{\{ H(x)<\infty \}}.
\end{align*}
\end{lem}

\begin{proof}[\bf Proof of Proposition~\ref{prop:trunc}]
Recall $V=[-\Cr{3.5}\|x\|_1,\Cr{3.5}\|x\|_1]^d$, and then let $V \cap \Z^d=\{ x_1,\dots,x_M \}$.
To shorten notation, set $\omega_i:=\omega(x_i)$ for $1 \leq i \leq M$.
Since $\tilde{a}(0,x,\omega)$ depends only on configurations in $V$, we can write
\begin{align*}
	f(\omega_1,\dots,\omega_M ):=\tilde{a}(0,x,\omega).
\end{align*}
Then, one has
\begin{align*}
	0 \leq \tilde{a}(0,x,\omega)-\hat{a}(0,x,\omega)
	&= f(\omega_1,\dots,\omega_M )-f(\hat{\omega}_1,\dots,\hat{\omega}_M )\\
	&= \sum_{i=0}^{M-1} \bigl\{ f([\hat{\omega},\omega]_i)-f([\hat{\omega},\omega]_{i+1}) \bigr\},
\end{align*}
where $[\hat{\omega},\omega]_0:=\omega$ and
\begin{align*}
	[\hat{\omega},\omega]_i
	:= (\hat{\omega}_1,\dots,\hat{\omega}_i,\omega_{i+1},\dots,\omega_M),\qquad 1 \leq i \leq M.
\end{align*}
Thanks to (A3) in $d=2$ and transience of the simple random walk in $d \geq 3$,
Lemma~\ref{lem:rankone} shows that there exists a constant $\Cl{zerner}$ such that
the left side of \eqref{eq:ptrunc} is smaller than or equal to
\begin{align*}
	&\P \Biggl( \sum_{i=0}^{M-1}
		\bigl\{ f([\hat{\omega},\omega]_i )-f([\hat{\omega},\omega]_{i+1}) \bigr\} \geq u \Biggr)\\
	&\leq \P \Biggl( \sum_{i=1}^M \1{\{ \omega_i \not= \hat{\omega}_i \}}
			\bigl\{ \omega_i-\hat{\omega}_i+\Cr{zerner} \bigr\} \geq u \Biggr)\\
	&= \P \Biggl(
		\sum_{i=1}^M \1{\{ \omega_i>(4d/\gamma) \log \|x\|_1 \}}
		\biggl\{ \omega_i-\frac{4d}{\gamma}\log \|x\|_1+\Cr{zerner} \biggr\} \geq u \Biggr).
\end{align*}
Therefore, for sufficiently large $x \in \Z^d$, one has the following upper bound on the left side of \eqref{eq:ptrunc}:
\begin{align*}
	\P \Biggl( \sum_{i=1}^M \1{\{ \omega_i>(4d/\gamma)\log \|x\|_1 \}} \omega_i \geq u \Biggr).
\end{align*}
In addition, we can estimate this as follows:
\begin{align*}
	&e^{-(\gamma/2)u}\biggl\{ 1+\int_{(4d/\gamma)\log \|x\|_1}^\infty (e^{(\gamma/2) s}-1)
		\,\P (\omega (0) \in ds) \biggr\}^M\\
	&\leq e^{-(\gamma/2)u} \bigl( 1+\|x\|_1^{-2d} \E[e^{\gamma \omega(0)}] \bigr)^M\\
	&\leq \exp \Bigl\{ -\frac{\gamma u}{2}+M\|x\|_1^{-2d} \E[e^{\gamma \omega (0)}] \Bigr\}.
\end{align*}
Since $V=[-\Cr{3.5}\|x\|_1,\Cr{3.5}\|x\|_1]^d$, $M$ is of order $\|x\|_1^d$.
Hence, \eqref{eq:ptrunc} follows.
On the other hand,  \eqref{eq:pexpect} is an immediate consequence of \eqref{eq:ptrunc}.
\end{proof}

\subsection{Proof of Theorem~\ref{thm:conc}}\label{subsect:pf_conc}
For the proof of Theorem~\ref{thm:conc}, our main tool is the martingale method as in \cite{Kes93} or \cite{Szn96}.
Throughout this subsection, we always assume (A1).
In addition, suppose that (A3) is valid if $d=2$.

For any sufficiently large $x \in \Z^d$, one has
\begin{align*}
	&\P \bigl( a(0,x)-\E[a(0,x)] \geq t\|x\|_1^{1/2} \bigr)\\
	&\leq \P \biggl( \tilde{a}(0,x)-\E[\tilde{a}(0,x)] \geq \frac{t}{3} \|x\|_1^{1/2} \biggr)
		+\P \biggl( \E[\tilde{a}(0,x)]-\E[a_V(0,x)] \geq \frac{t}{3} \|x\|_1^{1/2} \biggr).
\end{align*}
Then, \eqref{eq:com_exp} of Proposition~\ref{prop:compare} shows that for $t$ large enough, this is equal to
\begin{align*}
	\P \biggl( \tilde{a}(0,x)-\E[\tilde{a}(0,x)] \geq \frac{t}{3}\ \|x\|_1^{1/2} \biggr).
\end{align*}
From Proposition~\ref{prop:trunc}, this is smaller than
\begin{align}\label{eq:ccompare}
	\Cr{3.9} \exp \Bigl\{ -\frac{\gamma}{18} t\|x\|_1^{1/2} \Bigr\}
	+\P \biggl( \hat{a}(0,x)-\E[\hat{a}(0,x)] \geq \frac{t}{9} \|x\|_1^{1/2} \biggr).
\end{align}

To estimate the last probability, we will prepare some notation and lemmata.
For the enumerations $x_i$ and $\omega_i$ as in the proof of Proposition~\ref{prop:trunc},
let $\mathcal{F}_0$ be the trivial $\sigma$-field and $\mathcal{F}_i$ $\sigma$-field generated
by $\omega_1,\dots,\omega_i$.
Moreover, define
\begin{align*}
	\Delta_i:= \E[\hat{a}(0,x)|\mathcal{F}_i]-\E[\hat{a}(0,x)|\mathcal{F}_{i-1}],\qquad 1 \leq i \leq M.
\end{align*}

\begin{lem}\label{lem:kes-szn1}
Assume (A1).
In addition, suppose that (A3) is valid if $d=2$.
Then, there exist some constants $\Cl{3.10}$, $\Cl{3.11}$ independent of $x$ such that for all $1 \leq i \leq M$,
\begin{align}\label{eq:difbd}
	|\Delta_i| \leq \Cr{3.10} \log \|x\|_1
\end{align}
and
\begin{align}\label{eq:marti}
\begin{split}
	\E[\Delta_i^2 |\mathcal{F}_{i-1}]
	&\leq \Cr{3.11} \E \Bigl[ Q_{\hat{\omega}}^{0,x}(H(x)>H(x_i)) \Big| \mathcal{F}_{i-1} \Bigr].
\end{split}
\end{align}
\end{lem}

\begin{lem}\label{lem:kes-szn2}
Assume (A1) and let
\begin{align*}
	U_i:=\Cr{3.11} Q_{\hat{\omega}}^{0,x}(H(x)>H(x_i)).
\end{align*}
Then, there exist some constants $\Cl{ksxo}$ and $\Cl{3.12}$ independent of $x$
such that $\Cr{ksxo}>e^2\Cr{3.10}^2$, and for all $u \geq \Cr{ksxo}\|x\|_1$,
\begin{align*}
	\P \Biggl( \sum_{i=1}^M U_i \geq u \Biggr) \leq e^{-\Cr{3.12}u}.
\end{align*}
\end{lem}

Let us postpone the proofs of these lemmata.
Given Lemmata~\ref{lem:kes-szn1} and \ref{lem:kes-szn2}, we can apply (1.29) of \cite[Theorem~3]{Kes93}.
Then, there exist some constants $\Cl{3.13}$, $\Cl{3.14}$, $\Cl{3.15}$ such that for $t \leq \Cr{3.13}\|x\|_1$,
\begin{align}\label{eq:hat_cost}
	\P \biggl( \hat{a}(0,x)-\E[\hat{a}(0,x)] \geq \frac{t}{9} \|x\|_1^{1/2} \biggr)
	\leq \Cr{3.14} e^{-\Cr{3.15}t}.
\end{align}
Note that
\begin{align*}
	\hat{a}(0,x) \leq \biggl( \frac{4d}{\gamma}+\log(2d) \biggr) \|x\|_1 \log \|x\|_1,
\end{align*}
so that the left side of \eqref{eq:hat_cost} is equal to zero
for $t>9\{ 4d/\gamma+\log(2d) \}\|x\|_1^{1/2} \log\|x\|_1$.
Since $9\{ 4d/\gamma+\log(2d) \}\|x\|_1^{1/2} \log\|x\|_1<\Cr{3.13}\|x\|_1$ for all large $x$,
\eqref{eq:hat_cost} holds for all $t \geq 0$.
The theorem immediately follows from this and \eqref{eq:ccompare}.\qed

\begin{proof}[\bf Proof of Lemma~\ref{lem:kes-szn1}]
We can represent $\Delta_i$ as
\begin{align*}
	\Delta_i=\int_\Omega \bigl\{ \hat{a}(0,x, [\omega,\omega']_i )-\hat{a}(0,x, [\omega,\omega']_{i-1}) \bigr\}
	\,\P(d\omega').
\end{align*}
Thus, Schwarz's inequality shows
\begin{align}\label{eq:schwarz}
	\Delta_i^2
	\leq \int_\Omega \bigl\{ \hat{a}(0,x, [\omega,\omega']_i)
	-\hat{a}(0,x, [\omega,\omega']_{i-1}) \bigr\}^2 \,\P(d\omega' )
\end{align}
Lemma~\ref{lem:rankone} proves that
\begin{align}\label{eq:travel_rank}
\begin{split}
	&\bigl| \hat{a}(0,x, [\omega,\omega']_i)-\hat{a}(0,x, [\omega,\omega']_{i-1}) \bigr|\\
	&\leq \max_{s=\hat{\omega}_i,\hat{\omega}'_i} \Bigl[
		s+\bigl( 1-\min\{ e^{-s},P^0(H_2(0)<\infty) \} \bigr)^{-1} \Bigr].
\end{split}
\end{align}
Thanks to (A3) in $d=2$ and transience of the simple random walk in $d \geq 3$,
the last term is bounded from above by $\max\{ \hat{\omega}_i,\hat{\omega}'_i \}+\Cr{3.16}$
for some constant $\Cl{3.16}$.
Therefore,
\begin{align*}
	\Delta_i^2
	\leq \biggl( \frac{4d}{\gamma}\log\|x\|_1+\Cr{3.16} \biggr)^2 \vee \E[(\omega(0)+\Cr{3.16})^2],
\end{align*}
and \eqref{eq:difbd} holds for all large $x \in \Z^d$.

We next show \eqref{eq:marti}.
By \eqref{eq:schwarz}, $\E[\Delta_i^2|\mathcal{F}_{i-1}]$ is bounded by
\begin{align*}
	\int_\Omega \int_\Omega \bigl\{ \hat{a}(0,x,[\omega,\omega']_i)
	-\hat{a}(0,x,[\omega,\omega']_{i-1}) \bigr\}^2 \,\P(d\omega') \P(d\omega_i).
\end{align*}
Moreover, by symmetry we can restrict the domain of the integration above to those configurations with
$\omega'_i \leq \omega_i$.
Consequently, 
\begin{align*}
	&\E[\Delta_i^2|\mathcal{F}_{i-1}]\\
	&\leq 2\int_\Omega \int_\Omega \bigl\{ \hat{a}(0,x,[\omega,\omega']_i)
		-\hat{a}(0,x,[\omega,\omega']_{i-1}) \bigr\}^2 \1{\{ \omega'_i \leq \omega_i \}} \,\P(d\omega') \P(d\omega_i).
\end{align*}
Apply Lemma~\ref{lem:rankone} and \eqref{eq:travel_rank} again, and the integration in the right side is smaller than
or equal to
\begin{align*}
	&\E[(\omega (0) +\Cr{3.16})^2] \P \bigl(
		Q_{\hat{\omega}}^{0,x}(H(x) \leq H(x_i))<1/2 \big| \mathcal{F}_{i-1} \bigr)\\
	&+\E \Bigl[ \bigl\{ \log Q_{\hat{\omega}}^{0,x}(H(x) \leq H(x_i)) \bigr\}^2
		\1{\{ Q_{\hat{\omega}}^{0,x}(H(x) \leq H(x_i)) \geq 1/2 \}} \Big| \mathcal{F}_{i-1} \Bigr].
\end{align*}
Write $I_1$ and $I_2$ for the first and second terms  above, respectively.
Then,
\begin{align*}
	I_1
	&= \E\bigl[ (\omega (0) +\Cr{3.16})^2 \bigr]
		\P \bigl( Q_{\hat{\omega}}^{0,x}(H(x)>H(x_i))>1/2 \big| \mathcal{F}_{i-1} \bigr)\\
	&\leq 2\E\bigl[ (\omega (0) +\Cr{3.16})^2 \bigr]
		\E \bigl[ Q_{\hat{\omega}}^{0,x}(H(x)>H(x_i)) \big| \mathcal{F}_{i-1} \bigr].
\end{align*}
Using $(\log t)^2 \leq 1-t$ for $t \in [1/2,1]$, we have
\begin{align*}
	I_2 \leq \E \bigl[  Q_{\hat{\omega}}^{0,x}(H(x)>H(x_i)) \big| \mathcal{F}_{i-1} \bigr],
\end{align*}
and therefore
\begin{align*}
	I_1+I_2
	\leq (2\E\bigl[ (\omega (0) +\Cr{3.16})^2 \bigr] +1)
		\E \bigl[ Q_{\hat{\omega}}^{0,x}(H(x)>H(x_i)) \big| \mathcal{F}_{i-1} \bigr].
\end{align*}
With these observations, taking $\Cr{3.11}$ large enough, we get bound \eqref{eq:marti}.
\end{proof}

\begin{proof}[\bf Proof of Lemma~\ref{lem:kes-szn2}]
By the definition of $U_i$, we have
\begin{align}\label{eq:kesten_conc}
\begin{split}
	\P \Biggl( \sum_{i=1}^M U_i \geq u \Biggr)
	\leq \P \Bigl( E_{Q_{\hat{\omega}}^{0,x}}[\# \mathcal{A}] \geq \Cr{3.11}^{-1}u \Bigr),
\end{split}
\end{align}
where $\mathcal{A}:=\{ S_k; 0 \leq k<H(x) \}$.
Let $c:=-\log \E[e^{ -(\omega (0) \wedge (4d/\gamma))}]$.
We use Jensen's, Chebyshev's and Schwarz's inequalities to obtain that for $\gamma' \in (0,\gamma \wedge 1/2)$,
the right side of \eqref{eq:kesten_conc} is smaller than or equal to
\begin{align*}
	&\P \biggl( a(0,x)+\log E^0 \biggl[ \exp \biggl\{ \sum_{z \in \mathcal{A}} (c-\hat{\omega} (z)) \biggr\} \biggr]
		\geq c\Cr{3.11}^{-1}u \biggr)\\
	&\leq \exp \Bigl\{ -\frac{c\gamma'}{\Cr{3.11}}u \Bigr\}
		\E \Biggl[ e^{\gamma' a(0,x)}
		E^0 \biggl[ \exp \biggl\{ \sum_{z \in \mathcal{A}} (c-\hat{\omega} (z)) \biggr\} \biggr]^{\gamma'} \Biggr]\\
	&\leq \exp \Bigl\{ -\frac{c\gamma'}{\Cr{3.11}}u \Bigr\}
		\E \bigl[ e^{2\gamma' a(0,x)} \bigr]^{1/2}
		\E \Biggl[ E^0 \biggl[ \exp \biggl\{ \sum_{z \in \mathcal{A}} (c-\hat{\omega} (z)) \biggr\}
		\biggr]^{2\gamma'}\Biggr]^{1/2}.
\end{align*}
Notice that
\begin{align*}
	\E \bigl[ e^{2\gamma' a(0,x)} \bigr]^{1/2}
	\leq e^{\gamma' \|x\|_1\log (2d)} \E \bigl[ e^{2\gamma' \omega (0)} \bigr]^{\|x\|_1/2},
\end{align*}
and if $\|x\|_1 \geq e$, then
\begin{align*}
	\E \bigl[ e^{c-\hat{\omega} (0)} \bigr]
	\leq \E \bigl[e^{c-(\omega (0) \wedge (4d/\gamma))} \bigr]=1.
\end{align*}
Hence Jensen's inequality yields that
\begin{align*}
	\E \Biggl[ E^0 \biggl[ \exp \biggl\{ \sum_{z \in \mathcal{A}} (c-\hat{\omega} (z))
	\biggr\} \biggr]^{2\gamma'} \Biggr]^{1/2}
	&\leq \E \biggl[ E^0 \biggl[ \exp \biggl\{ \sum_{z \in \mathcal{A}} (c-\hat{\omega} (z)) \biggr\} \biggr]
		\biggr]^{\gamma'}\\
	&= E^0 \Biggl[ \prod_{z \in \mathcal{A}} \E[ \exp \{ c-\hat{\omega} (0) \} ] \Biggr]^{\gamma'} \leq 1.
\end{align*}
With these observations, we have
\begin{align*}
	\P \Biggl( \sum_{k=1}^M U_k \geq u \Biggr)
	\leq \exp \biggl\{ -\frac{c\gamma'}{\Cr{3.11}}u
		+\gamma' \|x\|_1\log (2d) +\frac{\|x\|_1}{2} \log \E \bigl[ e^{2\gamma' \omega (0)} \bigr] \biggr\}.
\end{align*}
Taking
\begin{align*}
	u_0
	:= \max \biggl\{ e^2C_{10}^2 (\log \|x\|_1)^2,
		\|x\|_1 \biggl( \frac{2\Cr{3.11}}{c} \log(2d)
		+\frac{\Cr{3.11}}{c\gamma'} \log \E \bigl[ e^{2\gamma' \omega (0)} \bigr] \biggr) \biggr\},
\end{align*}
one has for $u \geq u_0$,
\begin{align*}
	\P \Biggl( \sum_{k=1}^M U_k \geq u \Biggr)
	\leq \exp \Bigl\{ -\frac{c\gamma'}{2\Cr{3.11}}u \Bigr\}.
\end{align*}
This completes the proof.
\end{proof}

\section{The Gaussian concentration for the lower tail}\label{sect:gaussian}

In this section, we prove Theorem~\ref{thm:gaussian}.
To this end, we basically follow from the strategy taken in
\cite[Theorem~{1.1}]{DamKub14_arXiv}.
First, in Subsection~\ref{subsect:ent} we introduce the entropy and estimate it from above by
using the size of the range of the random walk, see Proposition~\ref{prop:ent_la} below.
Then we give a bound for this size under the weighted measure $Q_\omega^{0,x}$ as in Lemma~\ref{lem:rankone},
see Proposition~\ref{prop:mean_la} and Corollary~\ref{cor:moment_la} below.
Finally, we shall complete the proof of Theorem~\ref{thm:gaussian} in Subsection~\ref{subsect:pfg}.

\subsection{The entropy}\label{subsect:ent}
As mentioned above, our approach is based on \cite{DamKub14_arXiv} for the Gaussian concentration.
Accordingly, let us introduce the entropy $Ent(X)$ of a nonnegative random variable $X$ with $\E[X]<\infty$,
following the notation used in \cite{DamKub14_arXiv}:
\begin{align*}
	Ent(X):=\E[X\log X]-\E[X]\log \E[X].
\end{align*}
We have $Ent(X) \geq 0$ by Jensen's inequality.
If $\mu$ is the probability measure defined by $\mu(A)=\E[X \1{A}]/\E[X]$,
then the entropy of $X$ is the relative entropy $H(\mu|\P)$ of $Q$ with respect to $\P$ up to a normalizing constant multiple:
\begin{align*}
	\frac{Ent(X)}{\E[X]}=H(\mu|\P).
\end{align*}
This means that the following standard properties are inherited from the relative entropy:
\begin{align}\label{eq:ent_sup}
	Ent(X)=\sup \bigl\{ \E[XW]; \E[e^W] \leq 1 \bigr\}
\end{align}
and
\begin{align}\label{eq:ent_sum}
	Ent(X) \leq \sum_{y \in \Z^d} \E[Ent_y(X)],
\end{align}
where $Ent_y(X)$ is the entropy of $X$ considered only as a function of $\omega(y)$
(with all other configurations fixed).

The main object of this subsection is to obtain the following bound for the entropy,
which is the counterpart to Proposition~{2.5} of \cite{DamKub14_arXiv}.

\begin{prop}\label{prop:ent_la}
Assume (A2).
In addition, suppose that (A3) is valid if $d=2$.
There exists a constant $\Cl{4.2}$ such that
for $\lambda \leq 0$,
\begin{align}\label{eq:ent_la}
 Ent(e^{\lambda a(0,x)})
 \leq \Cr{4.2} \lambda^2 \E \Bigl[ e^{\lambda a(0,x)} \hat{E}_\omega^{0,x}[\# \mathcal{A}] \Bigr],
\end{align}
where $\mathcal{A}:=\{S_k;0 \leq k <H(x)\}$ and $\hat{E}_\omega^{0,x}$ is the expectation
with respect to the probability measure $Q_\omega^{0,x}$ as in Lemma~\ref{lem:rankone}.
\end{prop}

To this end, we prepare some notation and lemmata.
For a given $x \in \Z^d$ we set $U:=a(0,x)$.
In addition, for $y \in \Z^d$ let $U_y$ be the variable $U$ in the configuration in which
the potential at $y$ is replaced by its independent copy $\omega'(y)$.
Write $\E_y$ and $\bar{\E}_y$ for averages only over the potential at $y$,
and over both the potential at $y$ and its independent copy, respectively.
The next lemma follows from the same strategy taken in Lemma~2.7 of \cite{DamKub14_arXiv}.
For the reader's convenience, we shall repeat the argument, since our model is different and the proof is not long.

\begin{lem}\label{lem:ent_bound}
We have for $\lambda \leq 0$ and $y \in \Z^d$,
\begin{align}\label{eq:ent_bound}
 Ent_y(e^{\lambda U})
 \leq \lambda^2 \bar{\E}_y \bigl[ e^{\lambda U}(U_y-U)^2_+ \bigr].
\end{align}
\end{lem}
\begin{proof}
Since $f(t):=t\log t$ is convex on $[0,\infty)$, one has
\begin{align*}
 f(\E_y[X]) \geq f'(a)(\E_y[X]-a)+f(a)
\end{align*}
for suitable $X$ and $a>0$.
This gives
\begin{align*}
 \E_y[f(X)]-f(\E_y[X])
 \leq \E_y[f(X)]-f(a)-\E_y[(X-a)f'(a)].
\end{align*}
We now apply this with $X=e^{\lambda U}$ and $a=e^{\lambda U_y}$,
and integrate over the independent copy of the potential at $y$ to obtain
\begin{align*}
 &\E_y[f(e^{\lambda U})]-f(\E_y[e^{\lambda U}])\\
 &\leq \bar{\E}_y \bigl[ f(e^{\lambda U})-f(e^{\lambda U_y})
       -(e^{\lambda U}-e^{\lambda U_y})f'(e^{\lambda U_y}) \bigr].
\end{align*}
This, together with the definition of $f$, shows
\begin{align*}
 Ent_y(e^{\lambda U})
 \leq \bar{\E}_y \bigl[ e^{\lambda U}g(\lambda (U_y-U)) \bigr],
\end{align*}
where $g(t):=e^t-t-1$.
Since $g(t)=g(t_+)+g(-t_-)$, the right side is equal to
\begin{align*}
	\bar{\E}_y \bigl[ e^{\lambda U}g(\lambda (U_y-U)_+) \bigr]
	+\bar{\E}_y \bigl[ e^{\lambda U_y} e^{\lambda (U-U_y)} g(-\lambda (U_y-U)_-) \bigr].
\end{align*}
By symmetry, the second term is equal to
\begin{align*}
	\bar{\E}_y \bigl[ e^{\lambda U} e^{\lambda (U_y-U)_+} g(-\lambda (U_y-U)_+) \bigr].
\end{align*}
With these observations, setting $h(t):=t(e^t-1)$, one has
\begin{align*}
	Ent_y(e^{\lambda U})
	\leq \bar{\E}_y \bigl[ e^{\lambda U} h(\lambda (U_y-U)_+) \bigr],
\end{align*}
and \eqref{eq:ent_bound} follows from the fact that $h(t) \leq t^2$ for $t \leq 0$.
\end{proof}

After the preparation above, let us prove Proposition~\ref{prop:ent_la}.

\begin{proof}[\bf Proof of Proposition~\ref{prop:ent_la}]
Combining Lemmata~\ref{lem:rankone} and \ref{lem:ent_bound}, one has
\begin{align*}
	Ent_y(e^{\lambda U})
	\leq \lambda^2 \bar{\E}_y \Bigl[ e^{\lambda U}
		\Bigl( \min \Bigl\{ -\log Q_{\omega}^{0,x}(H(x) \leq H(y)), \omega'(y) +\Cr{4.3} \Bigr\} \Bigr)^2 \Bigr]
\end{align*}
for some constant $\Cl{4.3}$.
The last expectation is bounded from above by
\begin{align*}
 I_1+I_2
 &:= \bar{\E}_y \Bigl[ e^{\lambda U}
     (\omega'(y) +\Cr{4.3})^2 \1{\{ Q_{\omega}^{0,x}(H(x) \leq H(y))<1/2 \}} \Bigr]\\
 &\quad \ 
     +\E_y \Bigl[ e^{\lambda U}
      \bigl( \log Q_{\omega}^{0,x}(H(x) \leq H(y)) \bigr)^2
      \1{\{ Q_{\omega}^{0,x}(H(x) \leq H(y)) \geq 1/2 \}} \Bigr].
\end{align*}
We have
\begin{align*}
 I_1
 &= \bar{\E}_y \Bigl[ e^{\lambda U}
    (\omega'(y) +\Cr{4.3})^2 \1{\{ Q_{\omega}^{0,x}(H(x)>H(y))>1/2 \}} \Bigr]\\
 &\leq 2 \E[(\omega(0) +\Cr{4.3})^2]
    \E_y \bigl[ e^{\lambda U} Q_{\omega}^{0,x}(H(x)>H(y)) \bigr].
\end{align*}
Since $(\log t)^2 \leq 1-t$ for $1/2 \leq t \leq 1$,
it holds that
\begin{align*}
 I_2
 \leq \E_y \bigl[ e^{\lambda U} Q_{\omega}^{0,x}(H(x)>H(y)) \bigr].
\end{align*}
Therefore,
\begin{align*}
 I_1+I_2
 \leq (2 \E[(\omega(0) +\Cr{4.3})^2]+1)
      \E_y \bigl[ e^{\lambda U} Q_{\omega}^{0,x}(H(x)>H(y)) \bigr].
\end{align*}
This combined with \eqref{eq:ent_sum} yields
\begin{align*}
 Ent(e^{\lambda U})
 &\leq \sum_{y \in \Z^d} \lambda^2 (2 \E[(\omega(0) +\Cr{4.3})^2]+1)
       \E \Bigl[ \E_y \bigl[ e^{\lambda U} Q_{\omega}^{0,x}(H(x)>H(y)) \bigr] \Bigr]\\
 &= (2 \E[(\omega(0) +\Cr{4.3})^2]+1) \lambda^2
    \E \Bigl[ e^{\lambda U} \hat{E}_{\omega}^{0,x}[\# \mathcal{A}] \Bigr],
\end{align*}
which proves \eqref{eq:ent_la}.
\end{proof}

\subsection{The size of random lattice animals}\label{subsect:la}

First of all, let us prepare some notation to mention the main object of this subsection.
Given a potential $\omega \in \Omega$, we say that a subset $A$ of $\Z^d$ is occupied
if there exists $z \in A$ such that $\omega(z) \geq \kappa$,
otherwise $A$ is empty.
Fix $\kappa>0$ satisfying $\P(\omega(0) \geq \kappa)>0$.
For sufficiently large $l \in 2\N$, we consider the cubes
\begin{align*}
 C(q):=\Bigl( lq+\Bigl[ -\frac{l}{2},\frac{l}{2} \Bigr)^d \Bigr) \cap \Z^d,\qquad q \in \Z^d,
\end{align*}
and set
\begin{align*}
	\mathcal{C}_x
	:= \{ q \in \Z^d; \exists q' \in \Z^d \text{ with } \| q'-q \|_\infty \leq 1 \text{ such that } x \in C(q') \}.
\end{align*}
We now define the successive times of travel of $(S_k)_{k=0}^\infty$
at $\ell^\infty$-distance $3l/4$ as follows:
\begin{align*}
 &\tau_0 :=0,\\
 &\tau_{i+1} :=\inf\{ k>\tau_i; \| S_k-S_{\tau_i} \|_\infty \geq 3l/4\},\qquad i \geq 0.
\end{align*}
Observe that $S_{\tau_i}$ and $S_{\tau_{i+1}}$ lie in neighboring cubes $C(q)$
for the neighboring relation $\| q-q' \|_\infty \leq 1$.
Furthermore, during the time interval $[\tau_i,\tau_{i+1}]$,
the walk cannot visit more than $3^d$ distinct cubes $C(q)$.
Therefore, if we denote
\begin{align*}
 \tilde{\mathcal{A}}:=\{ q \in \Z^d; \exists i \geq 0 \text{ such that }
 S_{\tau_i \wedge H(x)} \in C(q) \},
\end{align*}
then $\tilde{\mathcal{A}}$ is $Q_\omega^{0,x} \hyphen \as$ an $\ell^\infty$-lattice animal,
i.e., a finite $\ell^\infty$-connected set of $\Z^d$ with  the adjacency relation $\|v_1-v_2\|_\infty=1$, $v_1,v_2 \in \Z^d$.
The definition of $\tilde{\mathcal{A}}$ implies
\begin{align}\label{eq:animals}
 \# \mathcal{A}
 \leq (3l)^d \# \tilde{\mathcal{A}} \qquad Q_\omega^{0,x} \hyphen \as
\end{align}

The following proposition is the main object of this subsection.

\begin{prop}\label{prop:mean_la}
Let $\Omega':=\{ \omega \in \Omega;(-l/8,l/8)^d \text{ is occupied} \}$ and let
\begin{align*}
 \chi
 := \sup_{\substack{\| z \|_\infty \leq l/2\\ \omega \in \Omega'}}
    E^z \biggl[ \exp \biggl\{ -\sum_{k=0}^{\tau_1-1}\omega (S_k) \biggr\} \biggr].
\end{align*}
Then, the following hold:
\begin{enumerate}
 \item
  We have $\chi \in (0,1)$.
 \item
  There exist constants $\Cl{4.4}$, $\Cl{4.5}$, $\Cl{4.6}$, $\Cl{4.7}$ such that
  on an event $A_n$ with $\P(A_n^c) \leq \Cr{4.4}e^{-\Cr{4.5}n}$, for all $x \in \Z^d$,
  \begin{align*}
   \hat{E}_\omega^{0,x}\bigl[ \exp \{ \Cr{4.6}\# \mathcal{A} \} \bigr]
   \leq \Cr{4.7}\chi^{-n} e(0,x)^{-1}.
  \end{align*}
\end{enumerate}
\end{prop}

This proposition is the counterpart to \cite[Theorem~{1.3}]{Szn95}
and it has already been mentioned in \cite[Lemma~2]{Le13_arXiv} with only a comment on the proof.
(For the annealed setting, the similar result was obtained by Kosygina--Mountford~\cite{KosMou12}.)
However, we have to review its proof more carefully to find the event $A_n$.
We will give the detailed proof of Proposition~\ref{prop:mean_la} to the end of this subsection.

The next corollary is an immediate consequence of Proposition~\ref{prop:mean_la}.

\begin{cor}\label{cor:moment_la}
Let $\Cl{4.8}:=(2/\Cr{4.6})(\log\chi^{-1}+1)$ and define for $x \in \Z^d$,
\begin{align*}
 Y_x
 := \hat{E}_\omega^{0,x}[\#\mathcal{A}]
    \1{\{ \Cr{4.8}a(0,x)<\hat{E}_\omega^{0,x}[\#\mathcal{A}] \}}.
\end{align*}
Then, there exist constants $\Cl{4.9}$, $\Cl{4.10}$ such that for all large $x \in \Z^d$,
\begin{align*}
 \E[e^{\Cr{4.9} Y_x}] \leq \Cr{4.10}.
\end{align*}
\end{cor}
\begin{proof}
From Proposition~\ref{prop:mean_la},
\begin{align*}
 \P( Y_x \geq n)
 &\leq \Cr{4.4}e^{-(\Cr{4.5}/\Cr{4.8})n}\\
 &\quad
      +\P \Bigl( A_{\lceil n/\Cr{4.8} \rceil} \cap \Bigl\{ \hat{E}_\omega^{0,x}[\#\mathcal{A}] \geq n,\,
       \Cr{4.8}a(0,x)<\hat{E}_\omega^{0,x}[\#\mathcal{A}] \Bigr\} \Bigr),
\end{align*}
On the event appearing in the last probability,
\begin{align*}
 \Cr{4.6} \hat{E}_\omega^{0,x}[\#\mathcal{A}]
 &\leq \log \hat{E}_\omega^{0,x}[\exp \{ \Cr{4.6} \#\mathcal{A} \}]\\
 &\leq \log \Cr{4.7} +\frac{1}{\Cr{4.8}} (\log \chi^{-1}+1) \hat{E}_\omega^{0,x}[\#\mathcal{A}].
\end{align*}
By the choice of $\Cr{4.8}$, one obtains
\begin{align*}
 \Cr{4.6} \hat{E}_\omega^{0,x}[\#\mathcal{A}]
 \leq \log \Cr{4.7}+\frac{\Cr{4.6}}{2} \hat{E}_\omega^{0,x}[\#\mathcal{A}]
\end{align*}
or
\begin{align*}
 \hat{E}_\omega^{0,x}[\#\mathcal{A}] \leq \frac{2}{\Cr{4.6}} \log \Cr{4.7}.
\end{align*}
However, since $\hat{E}_\omega^{0,x}[\#\mathcal{A}] \geq \| x \|_1$,
the inequality fails to hold for all $x \in \Z^d$ with $\| x \|_1>(2/\Cr{4.6}) \log \Cr{4.7}$.
This means that $ \P( Y_x \geq n) \leq \Cr{4.4}e^{-(\Cr{4.5}/\Cr{4.8})n}$ for all large $x \in \Z^d$.
Putting $\Cr{4.9}:=\Cr{4.5}/(2\Cr{4.8})$, one has for all large $x \in \Z^d$,
\begin{align*}
 \E[e^{\Cr{4.9} Y_x}]
 \leq \int_0^\infty \P \bigl( Y_x \geq \Cr{4.9}^{-1}\log u \bigr) \,du
 \leq 1+\Cr{4.4}e^{\Cr{4.5}/\Cr{4.8}},
\end{align*}
and the proof is complete.
\end{proof}

\begin{proof}[\bf Proof of Proposition~\ref{prop:mean_la}]
(1) See page~7 of \cite{Le13_arXiv} for the upper bound.
Let us only prove the lower bound.
If $l \in 2\N$ is sufficiently large, then we can find $\omega_0 \in \Omega'$.
Letting $r$ be a path on $\Z^d$ from $0$ to a site outside $(-3l/4,3l/4)^d$, one has
\begin{align*}
 \chi
 \geq \exp \biggl\{ -\sum_{z \in r}\omega_0 (z) \biggr\} \Bigl( \frac{1}{2d} \Bigr)^{\# r}>0.
\end{align*}
(2) Let $A_n$ be the event that for any $\ell^\infty$-lattice animal $\Gamma$ containing $0$
of the size bigger than $n$,
\begin{align*}
 \sum_{q \in \Gamma} \1{\{ \text{$lq+(-l/8,l/8)^d$ is occupied} \}}
 \geq \frac{\# \Gamma}{2}.
\end{align*}
By the same strategy as in Lemma~\ref{lem:restrict}, we have for some constants $\Cr{4.4}$ and $\Cr{4.5}$,
\begin{align*}
 \P (A_n^c) \leq \Cr{4.4}e^{-\Cr{4.5}n}.
\end{align*}

For $z \in \Z^d$, define $Oc(z):=1$ if $q \in \Z^d$ with $z \in C(q)$ is not in $\mathcal{C}_x$
and $lq+(-l/8,l/8)^d$ is occupied, otherwise $Oc(z):=0$.
Furthermore, let us introduce
\begin{align*}
 &M_0:=1,\\
 &M_m:=\prod_{i=0}^{m-1} \chi^{-Oc(S_{\tau_i})\1{\{ \tau_{i+1}<H(x) \}}}
       \exp \Biggl\{ -\sum_{k=0}^{\tau_m \wedge H(x)-1}\omega(S_k) \Biggr\}, \qquad m \geq 1.
\end{align*}
We will prove that $(M_m)_{m=0}^\infty$ is an $(\mathcal{F}_{\tau_m})_{m=0}^\infty$-supermartingale
under $P^0$, where $\mathcal{F}_{\tau_m}$ is the $\sigma$-field associated to the stopping time $\tau_m$.
For the proof, we use the strong Markov property to obtain
\begin{align*}
 &E^0[M_{m+1}|\mathcal{F}_{\tau_m}]\\
 &= M_m \Biggl( \1{\{ \tau_m \geq H(x) \}}\\
 &\quad
    +\1{\{ \tau_m <H(x) \}} E^{S_{\tau_m}} \Biggl[ \chi^{-Oc(S_0)\1{\{ \tau_1 <H(x) \}}}
    \exp \Biggl\{ -\sum_{k=0}^{\tau_1 \wedge H(x)-1}\omega(S_k) \Biggr\}
    \Biggr] \Biggr).
\end{align*}
Observe that when $S_{\tau_m} \in \bigcup_{q \in \mathcal{C}_x} C(q)$,
$Oc(S_0)=0$ holds $P^{S_{\tau_m}} \hyphen \as$ and
\begin{align*}
 E^{S_{\tau_m}} \Biggl[ \chi^{-Oc(S_0)\1{\{ \tau_1<H(x) \}}}
    \exp \Biggl\{ -\sum_{k=0}^{\tau_1 \wedge H(x)-1}\omega(S_k) \Biggr\} \Biggr]
 \leq 1.
\end{align*}
On the other hand, when $S_{\tau_m} \not\in \bigcup_{q \in \mathcal{C}_x} C(q)$,
we have $\tau_1<H(x)$ $P^{S_{\tau_m}} \hyphen \as$, so that
\begin{align*}
 &E^{S_{\tau_m}} \Biggl[ \chi^{-Oc(S_0)\1{\{ \tau_1<H(x) \}}}
     \exp \Biggl\{ -\sum_{k=0}^{\tau_1 \wedge H(x)-1}\omega(S_k) \Biggr\} \Biggr]\\
 &= E^{S_{\tau_m}} \biggl[ \chi^{-Oc(S_0)}
    \exp \biggl\{ -\sum_{k=0}^{\tau_1-1}\omega(S_k) \biggr\} \biggr].
\end{align*}
If $lq+(-l/8,l/8)^d$ with $S_{\tau_m} \in C(q)$ is empty, then this is equal to
\begin{align*}
 E^{S_{\tau_m}} \biggl[ \exp \biggl\{ -\sum_{k=0}^{\tau_1-1}\omega(S_k) \biggr\} \biggr]
 \leq 1.
\end{align*}
If $lq+(-l/8,l/8)^d$ with $S_{\tau_m} \in C(q)$ is occupied, then
$\| S_{\tau_m}-lq \|_\infty \leq l/2$ and $(-l/8,l/8)^d$ is occupied
in the configuration $\omega(\cdot+lq)$.
Therefore the definition of $\chi$ implies
\begin{align*}
 E^{S_{\tau_m}} \biggl[ \chi^{-Oc(S_0)}
 \exp \biggl\{ -\sum_{k=0}^{\tau_1-1}\omega(S_k) \biggr\} \biggr]
 &= \chi^{-1} E^{S_{\tau_m}-lq} \biggl[ \exp \biggl\{ -\sum_{k=0}^{\tau_1-1}\omega(S_k+lq) \biggr\} \biggr]\\
 &\leq 1.
\end{align*}
With these observations,
\begin{align*}
 E^0[M_{m+1}|\mathcal{F}_{\tau_m}]
 \leq M_m \bigl( \1{\{ \tau_m \geq H(y) \}} +\1{\{ \tau_m <H(y) \}} \bigr)
 = M_m,
\end{align*}
which proves that $(M_m)_{m=0}^\infty$ is an $(\mathcal{F}_{\tau_m})_{m=0}^\infty$-supermartingale
under $P^0$.

Using Fatou's lemma and the martingale convergence theorem, we can see that $E^0[M_\infty] \leq 1$.
Fix $\omega \in A_n$.
Then, $Q_\omega^{0,x} \hyphen \as$ either $\# \tilde{\mathcal{A}} \leq n$ or
\begin{align*}
 \sum_{q \in \tilde{\mathcal{A}}}\1{\{ \text{$lq+(-l/8,l/8)^d$ is occupied} \}}
 \geq \frac{\#\tilde{\mathcal{A}}}{2}.
\end{align*}
Hence, by \eqref{eq:animals},
\begin{align*}
 &\hat{E}_\omega^{0,x} \Bigl[ \exp \Bigl\{ \frac{1}{2(3l)^d}
  (\log \chi^{-1}) \#\mathcal{A} \Bigr\} \Bigr]\\
 &\leq \hat{E}_\omega^{0,x} \Bigl[ \exp \Bigl\{ \half (\log \chi^{-1}) \# \tilde{\mathcal{A}} \Bigr\} \Bigr]\\
 &\leq \chi^{-n/2}e(0,x)^{-1}
       +\hat{E}_\omega^{0,x} \Biggl[ \exp \Biggl\{ (\log \chi^{-1})
        \sum_{q \in \tilde{\mathcal{A}}}\1{\{ \text{$lq+(-l/8,l/8)^d$ is occupied} \}}
        \Biggr\} \Biggr].
\end{align*}
Note that
\begin{align*}
 \sum_{q \in \tilde{\mathcal{A}}}\1{\{ \text{$lq+(-l/8,l/8)^d$ is occupied} \}}
 \leq \sum_{i=0}^\infty Oc(S_{\tau_i}) \1{\{ \tau_{i+1}<H(x) \}} +\#\mathcal{C}_x,
\end{align*}
and thus
\begin{align*}
 \hat{E}_\omega^{0,x} \Biggl[ \exp \Biggl\{ \log \chi^{-1}
 \sum_{q \in \tilde{\mathcal{A}}}\1{\{ \text{$lq+(-l/8,l/8)^d$ is occupied} \}}
 \Biggr\} \Biggr]
 &\leq E^0[M_\infty] \chi^{-\#\mathcal{C}_x}e(0,x)^{-1}\\
 &\leq \chi^{-\#\mathcal{C}_x}e(0,x)^{-1}.
\end{align*}
This implies
\begin{align*}
 \hat{E}_\omega^{0,x} \Bigl[ \exp \Bigl\{ \frac{1}{2(3l)^d}
 \log \chi^{-1} \#\mathcal{A} \Bigr\} \Bigr]
 \leq (\chi^{-n/2}+\chi^{-\#\mathcal{C}_x})e(0,x)^{-1}.
\end{align*}
Since $\#\mathcal{C}_x=3^d$ and $\chi \in (0,1)$, the proof is complete. 
\end{proof}

\subsection{Proof of Theorem~\ref{thm:gaussian}}\label{subsect:pfg}

For the proof, we follow the strategy taken in Section~3 of \cite{DamKub14_arXiv}.
Define for $\lambda \leq 0$,
\begin{align*}
 \psi(\lambda)
 := \log \E[e^{\lambda(a(0,x)-\E[a(0,x)])}]
 = \log \E[e^{\lambda a(0,x)}]-\lambda \E[a(0,x)].
\end{align*}
Note that $\psi(\lambda)$ is nonnegative by Jensen's inequality.

The following proposition plays a key role in the proof of Theorem~\ref{thm:gaussian}.

\begin{prop}\label{prop:herbst}
Assume (A2).
In addition, suppose that (A3) is valid if $d=2$.
There exists constants $\Cl{4.13}$, $\Cl{4.14}$ such that for all $\lambda \in [-\Cr{4.13},0]$
and for all large $x \in \Z^d$,
\begin{align}\label{eq:herbst}
 \psi(\lambda) \leq \Cr{4.14} \lambda^2 \| x \|_1.
\end{align}
\end{prop}
\begin{proof}
Let us first show that there exist constants $\Cr{4.13}$, $\Cr{4.14}$ such that
for all $\lambda \in [-\Cr{4.13},0]$ and for all large $x \in \Z^d$,
\begin{align}\label{eq:ent_cost}
 Ent(e^{\lambda a(0,x)}) \leq \Cr{4.14} \lambda^2 \| x \|_1 \E[e^{\lambda a(0,x)}].
\end{align}
To this end,
\begin{align*}
 \E \Bigl[ e^{\lambda a(0,x)} \hat{E}_\omega^{0,x}[\#\mathcal{A}] \Bigr]
 \leq \Cr{4.8}\E[e^{\lambda a(0,x)} a(0,x)]+\E[e^{\lambda a(0,x)}Y_x].
\end{align*}
By Chebyshev's association inequality (see \cite[Theorem~2.14]{BouLugMas13_book}),
the first term is bounded by
\begin{align*}
 \Cr{4.8}\E[e^{\lambda a(0,x)}] \E[a(0,x)].
\end{align*}
By choosing $W=\Cr{4.9} Y_x-\log \E[e^{\Cr{4.9} Y_x}]$ in \eqref{eq:ent_sup},
a standard relative entropy inequality derives the following bound for the second term:
\begin{align*}
	\Cr{4.9} \E[e^{\lambda a(0,x)}Y_x] \leq \E[e^{\lambda a(0,x)}]\log \E[e^{\Cr{4.9} Y_x}]+Ent(e^{\lambda a(0,x)}).
\end{align*}
With these observations, Proposition~\ref{prop:ent_la} enables us to see that
\begin{align*}
 \E \Bigl[ e^{\lambda a(0,x)} \hat{E}_\omega^{0,x}[\#\mathcal{A}] \Bigr]
 &\leq \Cr{4.8}\E[e^{\lambda a(0,x)}] \E[a(0,x)]+\Cr{4.9}^{-1}\E[e^{\lambda a(0,x)}]\log \E[e^{\Cr{4.9} Y_x}]\\
 &\quad
       +\Cr{4.9}^{-1}\Cr{4.2} \lambda^2 \E \Bigl[ e^{\lambda a(0,x)} \hat{E}_\omega^{0,x}[\# \mathcal{A}] \Bigr]
\end{align*}
or
\begin{align*}
 &\E \Bigl[ e^{\lambda a(0,x)} \hat{E}_\omega^{0,x}[\#\mathcal{A}] \Bigr]\\
 &\leq \biggl( 1-\frac{\Cr{4.2}\lambda^2}{\Cr{4.9}} \biggr)^{-1}
       \Bigl( \Cr{4.8}\E[a(0,x)]+\Cr{4.9}^{-1}\log\E[e^{\Cr{4.9}Y_x}] \Bigr) \E[e^{\lambda a(0,x)}].
\end{align*}
Thanks to Corollary~\ref{cor:moment_la}, there exists a constant $\Cl{4.17}$ such that
for all large $x \in \Z^d$,
\begin{align*}
 \Cr{4.8}\E[a(0,x)]+\Cr{4.9}^{-1}\log\E[e^{\Cr{4.9}Y_x}] \leq \Cr{4.17}\| x \|_1.
\end{align*}
Therefore, put $\Cr{4.13}:=\{ \Cr{4.9}/(2\Cr{4.2}) \}^{1/2}$ and
use Proposition~\ref{prop:ent_la} again to obtain that for all $\lambda \in [-\Cr{4.13},0]$
and for all large $x \in \Z^d$,
\begin{align*}
 Ent(e^{\lambda a(0,x)})
 &\leq \Cr{4.2} \Cr{4.17}\lambda^2 \biggl( 1-\frac{\Cr{4.2}\lambda^2}{\Cr{4.9}} \biggr)^{-1} \| x \|_1
       \E[e^{\lambda a(0,x)}]\\
 &\leq 2\Cr{4.2}\Cr{4.17}\lambda^2\| x \|_1 \E[e^{\lambda a(0,x)}],
\end{align*}
which implies \eqref{eq:ent_cost} by taking $\Cr{4.14}:=2\Cr{4.2}\Cr{4.17}$.

Let us prove \eqref{eq:herbst}.
It is clear in the case $\lambda=0$,
so that we treat the case $\lambda<0$.
Notice that
\begin{align*}
 \psi'(\lambda)
 =\frac{\E[a(0,x)e^{\lambda a(0,x)}]}{\E[e^{\lambda a(0,x)}]}-\E[a(0,x)].
\end{align*}
Since
\begin{align*}
 \lambda \psi'(\lambda) -\psi(\lambda)
 = \frac{Ent(e^{\lambda a(0,x)})}{\E[e^{\lambda a(0,x)}]},
\end{align*}
\eqref{eq:ent_cost} shows that for all $\lambda \in [-\Cr{4.13},0)$ and for all large $x \in \Z^d$,
\begin{align*}
 \lambda \psi'(\lambda) -\psi(\lambda)
 \leq \Cr{4.14} \lambda^2 \| x \|_1.
\end{align*}
Dividing this by $\lambda^2$, we have
\begin{align*}
 \frac{d}{d\lambda} \Bigl( \frac{\psi(\lambda)}{\lambda} \Bigr)
 = \frac{\lambda \psi'(\lambda) -\psi(\lambda)}{\lambda^2}
 \leq \Cr{4.14} \| x \|_1.
\end{align*}
Because $\psi(\lambda)/\lambda \nearrow 0$ as $\lambda \nearrow 0$,
this implies that for all $\lambda \in [-\Cr{4.13},0)$ and for all large $x \in \Z^d$,
\begin{align*}
 \frac{\psi(\lambda)}{\lambda} \geq \lambda \Cr{4.14}\| x \|_1,
\end{align*}
and the proof is complete.
\end{proof}

Once Proposition~\ref{prop:herbst} is proved, the proof of Theorem~\ref{thm:gaussian}
is straightforward as follows.

\begin{proof}[\bf Proof of Theorem~\ref{thm:gaussian}]
Proposition~\ref{prop:herbst} implies that for all $\lambda \in [-\Cr{4.13},0]$
and for all large $x \in \Z^d$,
\begin{align*}
 \P (a(0,x)-\E[a(0,x)] \leq -s\| x \|_1 )
 \leq e^{s\lambda \| x \|_1} e^{\psi(\lambda)}
 \leq e^{(s\lambda +\Cr{4.14} \lambda^2) \| x \|_1}.
\end{align*}
Set $\lambda=-s/(2\Cr{4.14})$ for the bound
\begin{align*}
 \P (a(0,x)-\E[a(0,x)] \leq -s\| x \|_1 )
 \leq \exp \biggl\{ -\frac{s^2}{4\Cr{4.14}}\| x \|_1 \biggr\},
 \qquad 0 \leq s \leq 2\Cr{4.13}\Cr{4.14}.
\end{align*}
Substituting $s=t\| x \|_1^{-1/2}$, one has for $t \in [0,2\Cr{4.13}\Cr{4.14} \| x \|_1^{1/2}]$,
\begin{align*}
 \P \bigl( a(0,x)-\E[a(0,x)] \leq -t\| x \|_1^{1/2} \bigr)
 \leq \exp \biggl\{ -\frac{t^2}{4\Cr{4.14}} \biggr\}.
\end{align*}
For $t>2\Cr{4.13}\Cr{4.14} \| x \|_1^{1/2}$, we choose a constant
$\Cl{4.18}$ satisfying $\Cr{4.18} \| x \|_1 \geq \E[a(0,x)]$.
Then, for $t \in (2\Cr{4.13}\Cr{4.14} \| x \|_1^{1/2},\Cr{4.18}\| x \|_1^{1/2}]$,
\begin{align*}
 &\P \bigl( a(0,x)-\E[a(0,x)] \leq -t\| x \|_1^{1/2} \bigr)\\
 &\leq \P (a(0,x)-\E[a(0,x)] \leq -2\Cr{4.13}\Cr{4.14}\| x \|_1)
 \leq e^{-\Cr{4.13}^2\Cr{4.14}\| x \|_1},
\end{align*}
which is bounded by $\exp\{ -(\Cr{4.13}^2\Cr{4.14}/\Cr{4.18}^2) t^2 \}$.
In addition, for $t>\Cr{4.18}\| x \|_1^{1/2}$,
\begin{align*}
 &\P \bigl( a(0,x)-\E[a(0,x)] \leq -t\| x \|_1^{1/2} \bigr)\\
 &\leq \P \bigl( \E[a(0,x)] \geq t\| x \|_1^{1/2} \bigr)
  \leq \P(\E[a(0,x)]>\Cr{4.18}\| x \|_1) =0.
\end{align*}
Therefore the proof is complete.
\end{proof}

\section*{Acknowledgements}
The author would like to express his profound gratitude to the reviewer for the very careful reading of the manuscript.
I also thank Makiko Sasada for a comment on Theorem~\ref{thm:conc} at Tokyo Probability Seminar.

\bibliographystyle{abbrv}
\bibliography{SRWRP_conc}

\end{document}